\documentclass[a4paper,11pt,twoside]{article}

\usepackage[pdftex]{graphicx}

\usepackage{geometry,color,graphicx,mathtools,fixltx2e,microtype,lastpage,latexsym,dsfont, amsmath,amsfonts,amssymb,amsthm}

\usepackage[pdftex,pagebackref=false]{hyperref}

\usepackage[pagebackref=false]{hyperref}

\usepackage{geometry,amsmath}
\usepackage{mathtools}
\usepackage{mathrsfs}
\usepackage{amssymb}
\usepackage{bm}
\usepackage{graphicx}
\usepackage[latin1]{inputenc}
\usepackage{amsfonts}
\usepackage{amsthm}
\usepackage{thmtools}
\usepackage{newlfont}
\usepackage{latexsym}
\usepackage{mathrsfs}
\usepackage[nottoc,notlot,notlof]{tocbibind}
\usepackage{bbm}
\usepackage{xcolor}
\usepackage{color}
\usepackage{titlesec}
\titlespacing{\section}{0pt}{*0.5}{*0.5}
\usepackage{parskip}
\usepackage{footmisc}

\usepackage{xr}
\makeatletter
\newcommand{\vast}{\bBigg@{3.5}}
\newcommand{\Vast}{\bBigg@{5}}
\makeatother

\theoremstyle{plain}
\newtheorem{thm}{Theorem}[section]
\newtheorem{cor}[thm]{Corollary}

\newtheorem{prop}[thm]{Proposition}

\theoremstyle{definition}

\newtheorem{notation}{Notation}[section]

\theoremstyle{remark}
\newtheorem{oss}{Remark}[section]

\renewcommand{\(}{\left(}
\renewcommand{\)}{\right)}
\renewcommand{\[}{\left[}
\renewcommand{\]}{\right]}
\renewcommand{\{}{\left\lbrace}
\renewcommand{\}}{\right\rbrace}

\newcommand{\norm}[1]{\left\lVert#1\right\rVert}
\newcommand{\abs}[1]{\left\lvert#1\right\rvert}

\newcommand{\eps}{\varepsilon}

\numberwithin{equation}{section}

\theoremstyle{definition}

\def\R{\mathbb{R}}
\def\P{\mathbb{P}}

\def\NN{\mathcal{N}}

\def\Z{\mathbb{Z}}

\def\F{\mathscr{F} }

\def\DD{\mathbb{D} }

\def\A{\mathscr{A}}

\def\BB{\mathscr{B} }

\def\to{\rightarrow}

\def\1{\mathbbm{1}}

\def\Var{\operatorname{Var}}
\def\Tr{\operatorname{Tr}}

\def\r{\right\rangle}
\def\l{\left\langle}

\titleformat{\paragraph}
{\normalfont\normalsize\bfseries}{\theparagraph}{1em}{}
\titlespacing*{\paragraph}
{0pt}{3.25ex plus 1ex minus .2ex}{1.5ex plus .2ex}

\begin{document}

\begin{center} {\bf\large An Improved Second Order Poincar\'e Inequality \\ for Functionals of Gaussian Fields}
\end{center}
 \begin{center}Anna Vidotto\footnote[1]{Università degli Studi di Roma Tor Vergata, Facoltà di Scienze Matematiche, Fisiche e Naturali, Dipartimento di Matematica\\
E-mail: vidottoanna@gmail.com 
 }
 \end{center}
\allowdisplaybreaks

\begin{abstract}  

We present an improved version of the second order Gaussian Poincar\'e inequality, firstly introduced in Chatterjee (2009) and Nourdin, Peccati and Reinert (2009). These novel estimates are used in order to bound distributional distances between functionals of Gaussian fields and normal random variables.
Several applications are developed, including quantitative CLTs for non-linear functionals of stationary Gaussian fields related to the Breuer-Major theorem, improving previous findings in the literature and obtaining presumably optimal rates of convergence.

\end{abstract}

  
\section{Introduction}

The aim of the present paper is to prove several new and refined \emph{second order Poincar\'e inequalities} for the normal approximation of general functionals of Gaussian fields, thus improving previous findings in the literature. Our main motivation is to tackle a problem left open in \cite{NPR:09}, namely proving second order estimates yielding presumably optimal rates of convergence for integral transforms of Gaussian subordinated fields
(see the discussion in \cite[Remark 4.3, Remark 6.2]{NPR:09}). 
In this paper, we will provide an explicit answer to such an open problem, by using a powerful tool, namely the Mehler representation of the Ornstein-Uhlenbeck semigroup, which was exploited in recent years to obtain second order Poincar\'e inequalities for Poisson and Rademacher functionals, providing presumably optimal rates of convergence (see \cite{LPS:16} and \cite{KRT:17}). 

We will illustrate our findings through a number of applications: to non-linear functionals of continuous-time and discrete-time Gaussian processes, including the example that led to the discussion in \cite[Remark 4.3, Remark 6.2]{NPR:09}, to non-linear positive functionals of Brownian sheets in arbitrary dimension, and, in order to show the flexibility of our results, to limit theorems for random matrices related to the Sinai and Soshnikov CLT \cite{SS:98}.


\subsection{Previous work and plan of the paper}

Let $N\sim \NN(0,1)$ be a standard Gaussian random variable. Second order Poincar\'e inequalities can be seen as an iteration of the so-called {\it Gaussian Poincar\'e inequality} (hence the name), which states that 
\begin{equation}\label{PG}
\Var f(N)\leq E\big[f'(N)^2\big]\,,
\end{equation}
for every differentiable function $f:\R\rightarrow\R$, a result that was discovered by J. Nash in \cite{Na:56} and then reproved by H. Chernoff in \cite{Ch:81}. 
The estimate \eqref{PG} implies that, if the random variable $f'(N)$ has a small $L^2$ norm, then $f(N)$ has small fluctuations.
The Gaussian Poincar\'e inequality holds in the much more general setting of functionals of Gaussian fields and associated Malliavin operators, see \cite{HP:95}:
\begin{equation}\label{PM}
\Var F\leq E\[\norm{DF}_H^2\],
\end{equation}
where $F=f(X)$ is a general functional of an isonormal Gaussian process $X$ over a Hilbert space $H$ and $D$ is the first Malliavin derivative (see Section \ref{GS} for rigorous definitions). 
The first version of a second order Poincar\'e inequality was presented in \cite{Ch:09}, where the author proved that
one can iterate \eqref{PG} in order to assess the total variation distance $d_{TV}$ between the law of $f(N)$ and the law of a Gaussian random variable with matching mean and variance. The precise result is the following (see Section \ref{GS} for the definition of total variation distance $d_{TV}$):

\bigskip

\begin{thm}[Second order Poincar\'e inequality -- \cite{Ch:09}]\label{ch_thm}
Let $X=(X_1,\dots,X_d)$ be a standard Gaussian vector in $\R^d$. Take any $f\in C^2(\R^d)$ and let $\nabla f$ and $\nabla^2f$ denote the gradient and Hessian of $f$. Suppose $f(X)$ has a finite fourth moment and let $\mu=E f(X)$, $\sigma^2=\Var f(X)$. Let $Z\sim\NN\(\mu,\sigma^2\)$, then
\begin{equation}\label{SPCh}
d_{TV}\(f(X),Z\)\leq\frac{2\sqrt{5}}{\sigma^2}\(E\norm{\nabla f(X)}_{\R^d}^4\)^{1/4}\(E\norm{\nabla^2 f(X)}_{op}^4\)^{1/4},
\end{equation}
where $\norm{\cdot}_{op}$ stands for the operator norm of $\nabla^2 f(X)$ regarded as a random $d\times d$ matrix.
\end{thm}

Soon after \cite{Ch:09}, the authors of \cite{NPR:09} pointed out that the finite-dimensional Stein-type inequalities leading to relation \eqref{SPCh} are special instances of more general estimates, which can be obtained by combining Stein's method and Malliavin calculus on an infinite-dimensional Gaussian space. In particular, in \cite{NPR:09} the following general version of \eqref{SPCh} is obtained, involving functionals of arbitrary infinite-dimensional Gaussian fields (precise definitions of the Sobolev space $\DD^{2,4}$, Malliavin derivatives $D^\alpha,\,\alpha=1,2$, and of insonormal Gaussian process will be given in Section \ref{GS}).

\bigskip

\begin{thm}[Second order Poincar\'e inequality -- \cite{NPR:09}] \label{SOPI-NPR-thm}
Let $X$ be an isonormal Gaussian process over some real separable Hilbert space $H$, and let $F=f(X) \in \DD^{2,4}$. Assume that $E\[F\]=\mu$ and $\Var F=\sigma^2$. Let $N\sim\NN\(\mu,\sigma^2\)$. Then,
\begin{equation}\label{SOPI-NPR}
d_{TV}\(F,N\)\leq\frac{\sqrt{10}}{\sigma^2}\(E\norm{DF}_{H}^4\)^{1/4}\(E\norm{D^2F}_{op}^4\)^{1/4}\,,
\end{equation} 
where $\norm{\cdot}_{op}$ stands for the operator norm of the random Hilbert-Schmidt operator $g \mapsto \l g, D^2 F\r_{H}$.\end{thm} 

As already discussed, the initial impetus for the present paper comes from the fact that (as described e.g. in Remark 4.3 of \cite{NPR:09}), once these inequalities are applied, they often give suboptimal rate of convergence. Indeed, since in most applications of interest it is not possible to compute directly the expectation involving the operator norm in both bounds \eqref{SPCh} and \eqref{SOPI-NPR}, one is forced to move farther away from the distance in distribution and use bounds on the operator norm instead of computing it directly. 
Our strategy in order to overcome this difficulty is to adapt to the Gaussian setting an approach recently developed in \cite{LPS:16}, which relies on the use of the so-called Mehler formula (see \eqref{Mehler}), where the authors prove second order Poincar\'e inequalities for Gaussian approximation of Poisson functionals, yielding presumably optimal rates in several geometric applications. 

The next theorem contains one of the \emph{abstract} estimates developed in the present paper -- see Theorem \ref{SOPIi} below for a complete statement. 

\bigskip

\begin{thm}
Let $H:=L^2(A, \mathscr{A},\mu)$, where $(A, \BB(A))$ is a Polish space endowed with its Borel $\sigma$-field and $\mu$ is a positive, $\sigma$-finite and non-atomic measure and let $F=f(X)\in\mathbb{D}^{2,4}$ be s.t. $E[F]=0$, $E[F^2]=\sigma^2$, where $X$ is an isonormal Gaussian process over $H$.
\newpage
If $N\sim\NN(0,\sigma^2)$, then
\begin{eqnarray}
d_{TV}(F,N)&\leq&\frac{2\,\sqrt{3}}{\sigma^2}\(\int_{A\times A}\{E\[\(\(D^2F\otimes_1D^2F\)(x,y)\)^2\]\}^{1/2} \right.\times\notag\\
 &&\qquad\qquad\qquad \times \{E\[\(DF(x)DF(y)\)^2\]\}^{1/2}d\mu(x)d\mu(y)\bigg)^{1/2}.\label{first_bound}
 \end{eqnarray}
\end{thm}

\begin{oss}
The fact that $H$ is a $L^2$ space is fundamental for our proof. However, we will also see that our results are general enough, in order to imply explicit bounds for all the common situations of interest, including non-linear functionals of finite Gaussian vectors with arbitrary covariance matrices.
\end{oss}

Our main abstract results are successfully applied to deduce -- often sharp -- new quantitative central limit theorems (QCLTs)
for the following models:
\begin{itemize} 

\item[(i)] In Section \ref{sec_NLFIGP}, we obtain presumably sharp QCLTs for non-linear functionals of stationary Gaussian fields including:
\begin{itemize}
\item the increment of a Brownian motion,
\item the centred Ornstein-Uhlenbeck process,
\item the increments of a fractional Brownian motion;
\end{itemize}
hence obtaining Breuer-Major type results as well as improving the suboptimal rates of convergence obtained in \cite{NPR:09}.

\item[(ii)] In Section \ref{sheet}, we obtain a certainly optimal bound for non-linear positive functionals of a Brownian sheet on $\R^n$, exploding around singularities in the domain of integration, see also Remark \ref{exact}. This result is a generalization of limit theorems studied, with different techniques, in \cite{NP:05} and \cite{NP:09c}.

\item[(iii)] In Section \ref{Wigner}, we obtain a QCLT for the trace of a power $p_n$ of a $n\times n$ Gaussian Wigner matrix, with $p_n \to \infty$ as $n \to \infty$ in such a way that $p_n=o(n^{4/15})$, see Theorem \ref{WignerTHM}. This example is closely related to results in the famous paper \cite{SS:98} as well as to the QCLT proved in \cite{Ch:09} (see the discussion at the beginning of Section \ref{Wigner}). Our findings can be seen as an improvement of these results in terms of speed of $p_n$, see Remark \ref{comparison}.

\end{itemize}

\bigskip

\begin{oss}
We again stress that the reason why our second order Poincar\'e inequalities allow us to get sharp rates of convergence relies on the fact that all the quantities in \eqref{first_bound} are \emph{directly} computable. This is not the case for inequalities \eqref{SPCh} and \eqref{SOPI-NPR} where the authors, in order to apply their results, have to bound the operator norm using Cauchy-Schwarz inequality, moving farther away from the distance in distribution.
\end{oss}

To conclude this section, we mention that the present paper is one of the latest instalments in a growing body of work, connecting limit theorems (including those of the stable type) for functionals of Gaussian fields, and variational techniques based on Malliavin calculus - see \cite{NP:05}, \cite{PT:08}, \cite{NP:09}, \cite{NN:10}, \cite{NP:12}, \cite{NNP:16}. See also \url{https://sites.google.com/site/malliavinstein/home} for a complete list.

{\bf Plan of the paper.} 
Our paper is organised as follows: in the next section we explain our general setting, providing all the basic ingredients that we will use thorough the paper. In Section 2 we present our main results, while Section 3  contains the proofs. In Section 4 we prove QCLTs for some non-linear functionals of Gaussian fields, in particular non-linear functionals of stationary Gaussian fields (including some Breuer-Major type results, see \cite{BM:83}) and non-linear positive functionals of a Brownian sheet on $\R^n$.
Finally, in Section 5, we present a QCLT for the trace of a power $p_n$ of a $n\times n$ Gaussian Wigner matrix (some technical proofs are contained in the auxiliary file Appendix A that the reader can find at \url{https://annavidotto.files.wordpress.com/2018/06/auxiliary_file-appendix.pdf}).


\subsection{General setting}\label{GS}

\subparagraph{Probability distances}\label{dist}

We will consider several notions of {\it distances} between the distributions of two random vectors $X,Y$ with values in $\R^m$, $m\geq1$ (see \cite[Appendix C]{NP:12} and the references therein for a complete discussion):
\begin{itemize}
\item[1.] The \emph{Kolmogorov distance} 
\begin{flalign}
&d_{Kol}(X,Y)=\sup_{z_1,\dots, z_m \in \mathbb{R}}\big| P\(X\in(-\infty,z_1]\times\cdots\times(-\infty,z_m]\) \nonumber\\
                     &\qquad\qquad\qquad\qquad\qquad\qquad-P\(Y\in(-\infty,z_1]\times\cdots\times(-\infty,z_m]\)\big|. \label{dKol_DEF}
\end{flalign}
\item[2.] The \emph{total variation distance} 
\begin{equation}\label{dTV_DEF}
d_{TV}(X,Y)=\sup_{B\in\mathscr{B}(\mathbb{R}^m)}\abs{P\(X\in B\)-P\(Y\in B\)}.
\end{equation}
\item[3.] The \emph{Wasserstein distance}  
\begin{equation}\label{dW_DEF}
d_{W}(X,Y)=\sup_{h \in \mathscr{H}}\abs{E\[h(X)\] - E\[h(Y)\]}\,,
\end{equation}
where $\mathscr{H}$ is the class of all functions $h : \R^m \rightarrow \R$ such that $\norm{h}_{Lip}\leq1$, with 
\begin{equation}\label{LipNorm_DEF}
\norm{h}_{Lip}=\sup_{x,y\in \mathbb{R}^m,x\neq y}\frac{\abs{h(x)-h(y)}}{\norm{x-y}_{\mathbb{R}^m}}.
\end{equation}
\end{itemize}
It is immediate to note that $d_{Kol}(\cdot,\cdot)\leq d_{TV}(\cdot,\cdot)$. Moreover, if $X$ is any real-valued random variable and 
$N \sim \mathcal{N}(0,1)$, then $d_{Kol}(X,N)\leq 2\sqrt{d_{W}(X,N)}$ (see, among others, \cite[Theorem 3.3]{CGS:11} and more generally \cite[Theorem 3.1]{APP:16}).


\subparagraph{Gaussian analysis and Malliavin calculus}\label{Malliavin}

We will now present the basic elements of Gaussian analysis and Malliavin calculus that are used in this paper. The reader is referred to the two monographs \cite{N:06} and \cite{NP:12} for further informations.

Let $H=L^2(A, \BB(A), \mu)$, where $(A, \BB(A))$ is a Polish space endowed with its Borel $\sigma$-field and $\mu$ is a positive, $\sigma$-finite and non-atomic measure. An \textit{isonormal Gaussian process} $X=\{X(h): h \in H\}$ over $H$ is a centered Gaussian family defined on some probability space $(\Omega, \mathscr{F}, \mathbb{P})$ such that $E[X(h)X(g)]=\langle g,h \rangle_H$ for every $h,g \in H$. We will always assume $\mathscr{F}=\sigma(X)$ and write $L^2(\Omega)$ instead of $L^2(\Omega, \mathscr{F}, \mathbb{P})$.

Let $\mathcal{S}$ denote the set of all random variables of the form 
\begin{equation} \label{smooth}
f(X(\phi_1),\dots,X(\phi_m)),
\end{equation}
where $m\geq1$, $f:\mathbb{R}^m\rightarrow\mathbb{R}$ is a $C^{\infty}$-function such that $f$ and all its partial derivatives have at most polynomial growth at infinity, and $\phi_i \in H, i=1,\dots,m$. 
Note that the space $\mathcal{S}$ is dense in $L^q(\Omega)$ for every $q\geq1$.
Let $F \in \mathcal{S}$ be of the form \eqref{smooth}, the \textit{Malliavin derivative} of $F$ is the element of $L^2(\Omega; H)$ defined by
\begin{equation}\label{Mder}
DF=\sum_{i=1}^m \frac{\partial f}{\partial x_{i}}(X(\phi_1),\dots,X(\phi_m))\,\phi_{i};
\end{equation}
while the \textit{second Malliavin  derivative} of $F$ is the element of $L^2(\Omega; H^{\odot 2})$ given by
\begin{equation}\label{Mder2}
D^2F=\sum_{i,j=1}^m \frac{\partial^{2} f}{\partial x_{i}\partial x_{j}}(X(\phi_1),\dots,X(\phi_m))\,\phi_{i}\,\phi_{j}\,,
\end{equation}
where $H^{\odot 2}$ is the second symmetric tensor power of $H$, so that $H^{\odot 2}=L_s^2(A^2, \BB(A^2), \mu^2)$ is the subspace of 
$L^2(A^2, \BB(A^2), \mu^2)$ whose elements are a.e. symmetric. 

For $\alpha=1,2$, the operator $D^\alpha$ is closable ($D^1:=D$), so we can extend the domain of $D^\alpha$ to the space $\mathbb{D}^{\alpha,p}$, $p\geq 1$, which is defined as the closure of $\mathcal{S}$ with respect to the norm
$$
\|F\|_{\mathbb{D}^{\alpha,p}}=\(E[|F|^p]+E[\|DF\|_{H}^p+E[\|D^2F\|_{H^{\otimes 2}}^p]\mathbbm{1}_{\{\alpha=2\}}\)^{1/p}.
$$
Plainly, $\mathbb{D}^{2,p}\subset \mathbb{D}^{1,p}$. We call $\mathbb{D}^{\alpha,p}$ the \textit{domain of $D^\alpha$ in $L^p(\Omega)$}.
The space $\mathbb{D}^{\alpha,2}$ is a Hilbert space with respect to the inner product 
$$
\langle F,G \rangle_{\mathbb{D}^{\alpha,2}}=E\[FG\]+E\[\langle DF,DG\rangle_{H}\]+E\[\langle D^2F,D^2G\rangle_{H^{\otimes 2}}\]\mathbbm{1}_{\{\alpha=2\}}\,.
$$

Note that the Malliavin derivative satisfies the following {\it chain rule}. Let $\psi:\mathbb{R}\rightarrow\mathbb{R}$ be a continuously differentiable function with bounded partial derivatives, then if $F\in \mathbb{D}^{1,2}$, $\psi(F) \in \mathbb{D}^{1,2}$ and we have that
$D\psi(F) =\psi'(F)\,DF$.

For $n\in \mathbb{N}\cup\{0\}$, we call $H_n(x)=(-1)^ne^{\frac{x^2}{2}}\frac{d^n}{dx^n}(e^{-\frac{x^2}{2}})$
the \textit{$n$-th Hermite polynomial}. 
For each $n\geq0$ we define 
$$\mathcal{H}_n=\overline{span\{H_n(X(h)), h \in H,\|h\|_H=1\}}^{\|\cdot\|_{L^2(\Omega)}}.$$ 
The space $\mathcal{H}_n$ is called the $nth$ \textit{Wiener chaos} of $X$.
Clearly, we have $\mathcal{H}_0=\mathbb{R}$ and $\mathcal{H}_1=X$. Moreover, it is well known that $\mathcal{H}_n \bot \mathcal{H}_m$ for every 
$n\neq m$ and thus that the sum $\bigoplus_{n=0}^{\infty}\mathcal{H}_n$ is direct in $L^2(\Omega)$.
By the density of polynomial functions, this implies that every random variable $F \in L^2(\Omega)$ admits a unique expansion of the type $F=E[F]+\sum_{n=1}^{\infty}F_n$ where $F_n \in \mathcal{H}_n$ and the series converges in $L^2(\Omega)$.

The \textit{Ornstein-Uhlenbeck semigroup} $(P_t)_{t\geq0}$ is defined for all $t\geq0$ and $F \in L^2(\Omega)$ by $P_t(F)=\sum_{p=0}^{\infty}e^{-pt}J_p(F) \in L^2(\Omega)$, where $J_p(F)=\operatorname{Proj}(F|\mathcal{H}_p)$ stands for the orthogonal projection of $F$ onto the $p$-th Wiener chaos.
One can prove that for every $t >0$ and every $q\geq1$, $P_t$ is a contraction on $L^q(\Omega)$, that is $E\[\abs{P_t\(F\)}^q\]\leq\norm{F}_{L^q(\Omega)}^{q}$,  for all $F\in L^q(\Omega)$.
Let $F\in L^1(\Omega)$, let $X'$ be an independent copy of $X$, and assume that $X$ and $X'$ are defined on the product probability space 
$(\Omega \times \Omega',\F\otimes\F',\P \times \P')$. Since $F$ is measurable with respect to $X$, we can write $F = f(X)$ with $f : \R^H \to \R$ a measurable mapping determined $\P \circ X^{-1}$ a.s.. We have the so-called {\it Mehler formula}
\begin{equation} \label{Mehler}
P_tF=E\[f(e^{-t}X+\sqrt{1-e^{-2t}}X')\big|X\], \quad t\geq0\,.
\end{equation}

The generator $L$ of the Ornstein-Uhlenbeck semigroup is defined as $LF=-\sum_{p=1}^{\infty}pJ_p(F)$ with domain  given by 
$\operatorname{Dom}L=\{F \in L^2(\Omega):\sum_{p=1}^{\infty}p^2E\[J_p(F)^2\] < \infty\}$.
For any $F \in L^2(\Omega)$ we define $L^{-1}F=-\sum_{p=1}^{\infty}\frac{1}{p}J_p\(F\)$. The operator $L^{-1}$ is called the \textit{pseudo-inverse} of $L$. The name of $L^{-1}$ is justified by the fact that for any $F \in L^2(\Omega)$, $L^{-1}F \in \operatorname{Dom}L$ and $LL^{-1}F=F-E(F)$.
Let $F \in \mathbb{D}^{1,2}$ with $E[F]=0$, then the following relation holds
\begin{equation} \label{2repr}
-DL^{-1}F=\int_0^{\infty}e^{-t}P_tDFdt=-(L-I\,)^{-1}DF.
\end{equation}

For every $1\leq m\leq n$, every $r=1,\dots,m$, every $f \in L^2(A^n, \BB(A^n), \mu^n)$ and every $g \in L^2(A^m, \BB(A^m),\mu^m)$ we define the $r$-th contraction $f\otimes_rg: A^{n+m-2r}\to \R$ by
\begin{flalign}\label{contraction}
&f\otimes_rg(y_1,\dots,y_{n+m-2r})=\int_{A^r}f(x_1,\dots,x_r,y_1,\dots,y_{m-r}) \times\notag\\
                                                    &\qquad\qquad\qquad\times g(x_1,\dots,x_r,y_{m-r+1},\dots,y_{m+n-2r})d\mu(x_1)\cdots d\mu(x_r).
\end{flalign}

We stress that for each $F\in \mathbb{D}^{2,2}$ there exist two measurable processes $Y:\Omega\times A\to \R$ and $Z:\Omega\times A\times A\to \R$ such that for almost each $(\omega,a,b)\in \Omega\times A\times A$, $DF(\omega,a)=Y(\omega,a)$ and $D^2F(\omega,a,b)=Z(\omega,a,b)$ (for a detailed discussion see \cite[Section 1.2.1]{N:06}); for the rest of the paper we will always identify $DF$ and $D^2F$ with $Y$ and $Z$, respectively.


\section{Main results}

\subsection{Main estimates}

Let the notation of Section \ref{GS} prevail. Our main abstract result is the following. 

\bigskip

\begin{thm}\label{SOPIi}
Let $F\in\mathbb{D}^{2,4}$ be such that $E[F]=0$ and $E[F^2]=\sigma^2$, and let $N\sim\NN(0,\sigma^2)$; then
\begin{eqnarray}
d_{M}(F,N)&\leq& c_M\,\(\int_{A\times A}\{E\[\(\(D^2F\otimes_1D^2F\)(x,y)\)^2\]\}^{1/2} \right.\times\notag\\
 &&\qquad\qquad \times \{E\[\(DF(x)DF(y)\)^2\]\}^{1/2}d\mu(x)d\mu(y)\bigg)^{1/2},\label{SOPTV}
\end{eqnarray}
where $M \in \{TV,Kol,W\}$ and $c_{TV}=\frac{4}{\sigma^2}$, $c_{Kol}=\frac{2}{\sigma^2}$, $c_{W}=\sqrt{\frac{8}{\sigma^2\pi}}$.
\end{thm}

\subsection{Corollaries and extensions}

Theorem \ref{SOPIi} contains, as a special case, probabilistic approximations involving random variables of the form $F=f\(X_1,\dots,X_d\)$, where $\(X_1,\dots,X_d\)^T$ is a standard Gaussian vector and $f:\R^d\to\R$ is a $C^2$ function such that its partial derivatives have sub-exponential growth. Indeed, if $A_1,\dots,A_d\in \BB(A)$ are such that $A_i\cap A_j=\emptyset$ for each $i,j$ such that $i\neq j$ and $\mu(A_i)=1$ for all $i$, then we have that  
\begin{eqnarray*}
F&\stackrel{law}{=}&f\(X\(\mathbbm{1}_{A_1}\),X\(\mathbbm{1}_{A_2}\),\dots,X\(\mathbbm{1}_{A_d}\)\)\,.
\end{eqnarray*}
Moreover, in view of \eqref{Mder} and \eqref{Mder2}, we have that
$$
DF(x)=\sum_{i=1}^d\nabla_{i}f(X)\mathbbm{1}_{A_i}(x) \quad\text{and}\quad
D^2F(x,y)=\sum_{i,j=1}^d\nabla^2_{ij}f(X)\mathbbm{1}_{A_i}(x)\mathbbm{1}_{A_j}(y)\,,
$$ 
where $\nabla_{i}f(X)$ is the $i$-th component of the gradient of $f$ and $\nabla^2_{ij}$ is the $ij$-th entry of the Hessian matrix of $f$.
This implies that
\begin{flalign*}
D^2F\otimes_1D^2F(x,y)&=\int_A\,d\mu(w)\,\sum_{i,j=1}^d\nabla^2_{ij}f(X)\mathbbm{1}_{A_i}(x)\mathbbm{1}_{A_j}(w)\,\sum_{k,l=1}^d\nabla^2_{kl}f(X)\mathbbm{1}_{A_k}(y)\mathbbm{1}_{A_l}(w)\\
&=\sum_{i,k=1}^d\(\sum_{l=1}^d\nabla^2_{il}f(X)\nabla^2_{kl}f(X)\)\mathbbm{1}_{A_i}(x)\mathbbm{1}_{A_k}(y)\,.
\end{flalign*}
In this case, the quantities on the right hand side of inequality \eqref{SOPTV} become, respectively,
\begin{flalign*}
&\{E\[\(\(D^2F\otimes_1D^2F\)(x,y)\)^2\]\}^{1/2}=\\
&=\{E\[\(\sum_{i,k=1}^d\(\sum_{l=1}^d\nabla^2_{il}f(X)\nabla^2_{kl}f(X)\)\mathbbm{1}_{A_i}(x)\mathbbm{1}_{A_k}(y)\)^2\]\}^{1/2}\\
&=\{\sum_{i,k=1}^dE\[\(\sum_{l=1}^d\nabla^2_{il}f(X)\nabla^2_{kl}f(X)\)^2\]\mathbbm{1}_{A_i}(x)\mathbbm{1}_{A_k}(y)\}^{1/2}\\
&=\sum_{i,k=1}^d\{E\[\(\sum_{l=1}^d\nabla^2_{il}f(X)\nabla^2_{kl}f(X)\)^2\]\}^{1/2}\mathbbm{1}_{A_i}(x)\mathbbm{1}_{A_k}(y)\\
\end{flalign*}
and, with similar steps,
\begin{flalign*}
&\{E\[\(DF(x)DF(y)\)^2\]\}^{1/2}=
\sum_{i,k=1}^d\{E\[\(\nabla_{i}f(X)\nabla_{k}f(X)\)^2\]\}^{1/2}\mathbbm{1}_{A_i}(x)\mathbbm{1}_{A_k}(y)\,.
\end{flalign*}
Hence, when $F=f\(X_1,\dots,X_d\)$, with $\(X_1,\dots,X_d\)$ a standard Gaussian vector, our main result takes the following form.

\bigskip

\begin{thm}\label{SOPIf}
Let $X=(X_1,\dots,X_d)\sim \mathcal{N}(0,I_{d\times d})$ and  $F:=f(X)$ for some $f \in C^2(\mathbb{R}^d)$ such that $E[F]=0$ and $E[F^2]=\sigma^2$. 
Let $N\sim\NN(0,\sigma^2)$, then
\begin{equation}\label{SOPfTV}
d_{M}(F,N)\leq c_M\,\sqrt{\sum_{i,l=1}^d \{E\[\(\sum_{j=1}^d\nabla^2_{ij}F\nabla^2_{lj}F\)^2\]\}^{1/2}\{E\[\(\nabla_{i}F\nabla_{l}F\)^2\]\}^{1/2}},
\end{equation}
where $M \in \{TV,Kol,W\}$, $c_{TV}=\frac{4}{\sigma^2},\, c_{Kol}=\frac{2}{\sigma^2},\,c_{W}=\sqrt{\frac{8}{\sigma^2\pi}}$ and $\nabla^2_{ij}F$ is the $ij$-th entry of the Hessian matrix of $F=f(X)$ while $\nabla_{i}F$ is the $i$-th element of the gradient of $F$.
\end{thm}

\bigskip

\begin{oss} \quad
Note that Theorem \ref{SOPIf} also applies to the case of a vector $X$ with a general covariance, that is $X \sim \mathcal{N}(0,B^2)$, where $B^2$ is a symmetric and positive definite $d\times d$ matrix. Indeed, one has that $F=f(X)=g(Z)$, where $g=f\circ B$ and $Z=(Z_1,\dots,Z_d)\sim \mathcal{N}(0,I_{d\times d})$. Therefore we have that
\begin{flalign*}
&d_{TV}(F,N)\leq\frac{4}{\sigma^2}\sqrt{\sum_{i,l=1}^d \{E\[\(\sum_{j=1}^d\nabla^2_{ij}g(Z)\nabla^2_{lj}g(Z)\)^2\]\}^{1/2}\{E\[\(\nabla_{i}g(Z)\nabla_{l}g(Z)\)^2\]\}^{1/2}}\\
&=\frac{4}{\sigma^2}\{\sum_{i,l=1}^d \{E\[\(\sum_{j=1}^d\sum_{k,m,r,s=1}^db_{mi}b_{kj}b_{rl}b_{sj}\nabla^2_{km}F\,\nabla^2_{rs}F\)^2\]\}^{1/2}\right.\times\\
&\left.\qquad\qquad\qquad\qquad\qquad\qquad\qquad\qquad\times\{E\[\(\sum_{k,m=1}^db_{ki}b_{ml}\nabla_{k}F\,\nabla_{m}F\)^2\]\}^{1/2}\}^{1/2}\,,
\end{flalign*}
with $b_{ij}$ the $ij$-th entry of the matrix $B$.
\end{oss}

Using the multidimensional version of \cite[Theorem 5.1.3]{NP:12} (which is one of the main ingredient of our main result's proof, see Section \ref{proof}), that is \cite[Theorem 6.1.1]{NP:12}, Theorem \ref{SOPIi} can be easily extended to a multidimensional setting as follows:

\bigskip

\begin{thm}\label{SOPImulti}
Let $F=(F_1,\dots,F_d)$, where, for each $i=1,\dots,d$, $F_i\in\mathbb{D}^{2,4}$ is such that $E[F_i]=0$ and $E[F_iF_j]=c_{ij}$, with $C=\{c_{ij}\}_{i,j=1,\dots,d}$ a symmetric and positive definite matrix. Let $N\sim\mathcal{N}(0,C)$, then we have that
\begin{flalign*}
&d_W(F,N)\leq2\sqrt{d}\norm{C^{-1}}_{op}\norm{C}_{op}\\
&\sqrt{\sum_{i,j=1}^d\int_{A\times A}\{E\[\(\(D^2F_i\otimes_1D^2F_i\)(x,y)\)^2\]\}^{1/2}\{E\[\(DF_j(x)DF_j(y)\)^2\]\}^{1/2}d\mu(x)d\mu(y)}.
\end{flalign*}
\end{thm}


\section{Proof of Theorem \ref{SOPIi}}\label{proof}

An important ingredient in order to prove our main result is a theorem given in \cite[Theorem 5.1.3]{NP:12} and, with slight more generality, in \cite[Theorem 5.2]{Nou:13}.

\bigskip

\begin{thm}[\cite{NP:12}]\label{NP12}
Let $F\in \DD^{1,2}$ with $E\[F\]=0$ and $E\[F^2\]=\sigma^2$, and let $N\sim\NN(0,\sigma^2)$. Then,
$$
d_{M}\(F,N\)\leq c_M\,E\[\abs{\sigma^2-\langle DF,-DL^{-1}F\rangle_{H}}\],
$$
where $M \in \{TV,Kol,W\}$ and $c_{TV}=\frac{2}{\sigma^2},\, c_{Kol}=\frac{1}{\sigma^2},\,c_{W}=\sqrt{\frac{2}{\sigma^2\pi}}$.

\end{thm}

In order to prove Theorem \ref{SOPIi} we need a new crucial intermediate result, given in the following proposition. 

\bigskip

\begin{prop}\label{mi}
Let $F,G \in \mathbb{D}^{2,4}$ such that $E[F]=E[G]=0$. Then, it holds that
\begin{flalign*}
&E\[\(\operatorname{Cov}(F,G)-\langle DF,-DL^{-1}G\rangle_{L^2(A,\mu)}\)^2\] \leq \\ 
& \leq 2\,\int_{A\times A} \{E\[\(\(D^2F\otimes_1D^2F\)(x,y)\)^2\]\}^{1/2}\{E\[\(DG(x)DG(y)\)^2\]\}^{1/2}d\mu(x)d\mu(y)+ \\
&\quad+2\,\int_{A\times A}\{E\[\(DF(x)DF(y)\)^2\]\}^{1/2}\{E\[\(\(D^2G\otimes_1D^2G\)(x,y)\)^2\]\}^{1/2} d\mu(x)d\mu(y).
\end{flalign*}
\end{prop}

\proof
Using the fact that $\operatorname{Cov}(F,G)=E\(\langle DF,-DL^{-1}G\rangle_{L^2(A,\mu)}\)$ and the Poincar\'e inequality \eqref{PM} (note that one needs $F,G\in \DD^{2,4}$ for $\langle DF,-DL^{-1}G\rangle_{L^2(A,\mu)}$ to be in $\DD^{1,2}$ and apply \eqref{PM}, see \cite[Lemma 3.2]{NPR:09}), we have 
\begin{flalign}
&E\[\(\operatorname{Cov}(F,G)-\langle DF,-DL^{-1}G\rangle_{L^2(A,\mu)}\)^2\] =\operatorname{Var}(\langle DF,-DL^{-1}G\rangle_{L^2(A,\mu)})\notag\\
&\leq E\(\norm{D\langle DF,-DL^{-1}G\rangle_{L^2(A,\mu)}}_{L^2(A,\mu)}^2\)\notag\\
&\leq2E\(\underbrace{\norm{\langle D^2F,-DL^{-1}G\rangle_{L^2(A,\mu)}}_{L^2(A,\mu)}^2}_{A_1}\)+2E\(\underbrace{\norm{\langle DF,-D^2L^{-1}G\rangle_{L^2(A,\mu)}}_{L^2(A,\mu)}^2}_{A_2}\), \label{A1A2}
\end{flalign}
where the last inequality follows from the fact that (again, according to \cite[Lemma 3.2]{NPR:09})
$$
D\langle DF,-DL^{-1}G\rangle_{L^2(A,\mu)}=\langle D^2F,-DL^{-1}G\rangle_{L^2(A,\mu)}+\langle DF,-D^2L^{-1}G\rangle_{L^2(A,\mu)}\,.
$$
Let us first consider $A_1$: given the fact that (see \eqref{2repr})
$$
-DL^{-1}G=\int_0^{\infty}e^{-t} P_tDGdt
$$
and using Mehler formula \eqref{Mehler}, we deduce that
\begin{flalign}
&\langle D^2F,-DL^{-1}G\rangle_{L^2(A,\mu)}=\langle D^2F, \int_0^{\infty}e^{-t} P_tDGdt\rangle_{L^2(A,\mu)}\label{step1_meh}\\
                                                                   &=\langle D^2F, \int_0^{\infty}e^{-t} E\(Dg\(e^{-t}X+\sqrt{1-e^{-2t}}X'\)\Big|X\)dt\rangle_{L^2(A,\mu)}\label{step2_meh}\\ 
                                                                               &=\int_0^{\infty}e^{-t} E\[\langle D^2F, Dg\(e^{-t}X+\sqrt{1-e^{-2t}}X'\)\rangle_{L^2(A,\mu)}\Big|X\]dt\,.\notag
\end{flalign}
Hence, Jensen inequality and Fubini theorem yield that
\begin{flalign*}
A_1&=\norm{\int_0^{\infty}e^{-t} E\[\langle D^2F, Dg\(e^{-t}X+\sqrt{1-e^{-2t}}X'\)\rangle_{L^2(A,\mu)}\Big|X\]dt}_{L^2(A,\mu)}^2\\
      &\leq\int_0^{\infty}e^{-t}E\[\norm{\langle D^2F, Dg\Big(\underbrace{e^{-t}X+\sqrt{1-e^{-2t}}X'}_{X_t}\Big)\rangle_{L^2(A,\mu)}}^2\Bigg|X\]dt\\
       &= \int_0^{\infty}e^{-t}E\[\norm{\int_A(D^2F)(x,y)Dg(X_t)(x)d\mu(x)}_{L^2(A,\mu)}^2\Bigg|X\]dt\\
&=\int_0^{\infty}e^{-t}\int_{A^2}\int_AD^2F(x,y)D^2F(z,y)E\[Dg(X_t)(x)Dg(X_t)(z)\bigg|X\]d\mu(x)d\mu(z)d\mu(y)dt\\
&=\int_0^{\infty}e^{-t}\int_{A^2}\int_AD^2F(x,y)D^2F(z,y)P_t\(DG(x)DG(z)\)d\mu(x)d\mu(z)d\mu(y)dt\,.
\end{flalign*}
Now we can use Cauchy-Schwarz inequality and the contractivity of $P_t$ to have 
\begin{flalign*}
&E\(A_1\) \leq\int_0^{\infty}e^{-t}\int_{A\times A} \{E\[\(\int_AD^2F(x,y)D^2F(z,y)d\mu(y)\)^2\]\}^{1/2} \times\\
              & \qquad\qquad \qquad\qquad\times \{E\[\(P_t\(DG(x)DG(z)\)\)^2\]\}^{1/2}d\mu(x)\,d\mu(z)\,dt\\
              &\leq\int_{A\times A}\{E\[\(\(D^2F\otimes_1D^2F\)(x,z)\)^2\]\}^{1/2} \{E\[\(DG(x)DG(z)\)^2\]\}^{1/2}d\mu(x)d\mu(z)\,.
\end{flalign*}
Similarly, using again in order Mehler formula \eqref{Mehler}, Jensen inequality, Fubini theorem, Cauchy-Schwarz inequality and the contractivity of $P_t$, we also obtain that
\begin{flalign*}
E\(A_2\)
              &\leq\int_{A\times A}\{E\[\(\(D^2G\otimes_1D^2G\)(x,y)\)^2\]\}^{1/2} \{E\[\(DF(x)DF(y)\)^2\]\}^{1/2}d\mu(x)d\mu(y).
\end{flalign*}
Finally,
$$
E\[\(\operatorname{Cov}(F,G)-\langle DF,-DL^{-1}G\rangle_{L^2(A,\mu)}\)^2\] \leq 2E(A_1)+2E(A_2),
$$
which gives the desired conclusion.
\endproof

\begin{oss}\label{useofMehler}
The crucial difference between our main result Theorem \ref{SOPIi} and Theorem \ref{SOPI-NPR-thm} in \cite{NPR:09} can be found in the proof of Proposition \ref{mi}. 
Indeed, the authors of \cite{NPR:09} use Cauchy-Schwarz inequality to bound both $A_1$ and $A_2$ in \eqref{A1A2} by $E\norm{DF}_{H}^2 E\norm{D^2F}_{op}^2$ and to obtain their form of second order Poincar\'e inequality. On the contrary, we only use the Mehler representation of the Ornstein-Uhlenbeck semigroup $P_t$, as showed in steps \eqref{step1_meh}-\eqref{step2_meh}, in order to get a bound for $A_1$ and $A_2$ in terms of directly computable quantities. Indeed, the problem of suboptimal rates in \cite{NPR:09} relies on the fact that the operator norm of $D^2F$ is not directly computable.
\end{oss}

\begin{proof}[Proof of Theorem \ref{SOPIi}]
Taking $G=F$ in Proposition \ref{mi}, one has that
\begin{flalign}
&E\[\abs{1-\langle DF,-DL^{-1}F\rangle_{L^2(A,\mu)}}\]\leq\sqrt{E\[\(1-\langle DF,-DL^{-1}F\rangle_{H}\)^2\]}\notag\\
&\leq\,2\,\sqrt{\int_{A\times A}\mu(dx)\mu(dy)\{E\[\(DF(x)DF(y)\)^2\]\}^{1/2}\{E\[\(\(D^2F\otimes_1D^2F\)(x,y)\)^2\]\}^{1/2}}.\label{tz}
\end{flalign}
As a consequence, combining \eqref{tz} with Theorem \ref{NP12}, we immediately obtain our main result.
\end{proof}


\section{Applications to infinite-dimensional Gaussian fields}

In this section we will apply our main findings to the following models:
\begin{itemize} 

\item[{\bf4.1}] Non-linear functionals of stationary Gaussian fields (improving the suboptimal rates of convergence obtained in \cite{NPR:09} and obtaining Breuer-Major type results), including:
\begin{itemize}
\item the increment of a Brownian motion;
\item the centred Ornstein-Uhlenbeck process;
\item the increments of a fractional Brownian motion.
\end{itemize}

\item[{\bf4.2}] Non-linear positive functional of a Brownian sheet on $\R^n$, exploding around singularities in the domain of integration, which is a generalization of limit theorems studied, with different techniques, in \cite{NP:05} and \cite{NP:09c}.

\end{itemize}


\subsection{Non-linear functionals of an isonormal stationary Gaussian process}\label{sec_NLFIGP}

In this section, we use our results in order to assess the distance in distribution between a general non-linear functional of a stationary Gaussian process and a Gaussian random variable. 
This application includes, as a special instance, the example considered in \cite[Section 6]{NPR:09}, where a suboptimal rate of convergence was attained. As already underlined, this fact gave the initial impetus for the present paper: in this section we will indeed obtain a better rate of convergence which is presumably optimal (see also Remark \ref{suboptimal} later in the text).

Our starting point is the following general setting, which is flexible enough for many specific applications that will be developed later in the text.

Let $X=\{X(h): h \in H\}$ be an isonormal Gaussian process over the real separable Hilbert space $H=L^2\(\R,\BB\(\R\),\mu\)$. Let $A\subset\R$ be such that $0<\mu\(A\)<\infty$ and let $\{K_a:a\in A\}\subset H$ be such that the scalar product $\l K_a,K_b\r=\varrho(a-b)$, with $\varrho(0)=1$, only depends on the difference $a-b$, for every $a,b \in A$, with 
$$
\int_\R\abs{\varrho(a)}d\mu(a)<\infty\,.
$$
We define $\{Y_a=X(K_a):a\in A\}$ and assume that the mapping $(\omega, a)\mapsto Y_a(\omega)$ is jointly measurable. 

Let $f:\mathbb{R}\rightarrow\mathbb{R}$ be a real function of class $C^2$ such that $E\abs{f(N)}<\infty$ and $E\abs{f''(N)}^4<\infty$, with $N\sim\mathcal{N}(0,1)$ (which implies $E\abs{f(N)}^4,E\abs{f'(N)}^4<\infty$, via the classical Poincar\'e inequality). We can define the functional $F$ of $\(Y_a\)_{a\in A}$ in the following way 
$$
F=\frac{1}{\sqrt{\mu\(A\)}}\int_{A}\, f\(Y_a\)-E\[f\(Y_a\)\]\,d\mu(a)
$$
and our result goes as follows.

\medskip

\begin{prop}\label{NLFIGP} 
Assume that 
\begin{equation}\label{condition}
\abs{K_a(s)}\leq g(a-s)\,, \quad \text{where $g$ is s.t.} \quad G^{\star}:=\sup_{h\in\R}\int_{\mathbb{R}}g(t+h)\,d\mu(t)<\infty\,.
\end{equation}
Then, assuming $\Var F=\sigma^2>0$, 
$$
d_{TV}\(\frac{F}{\sigma},N\)\leq \frac{1}{\sigma^2}\cdot\frac{C}{\sqrt{\mu\(A\)}},
$$
where $N\sim\NN(0,1)$ and $C$ is a constant that does not depend on $\mu(A)$.
\end{prop}

\proof
Without loss of generality, let us set $\sigma=1$. By definition of Malliavin derivatives with respect to $X$ and thanks to the stochastic Fubini theorem (see \cite{Ve:12}), we have that 
\begin{equation}\label{Mder+int}
DF=\frac{1}{\sqrt{\mu(A)}}\int_{A}f'\(Y_a\)K_a(x)d\mu(a)\,,                                 
\quad
D^2F=\frac{1}{\sqrt{\mu(A)}}\int_{A}f''\(Y_a\)K_a(x)K_a(y)d\mu(a)\,,                              
\end{equation}
where we recall that $Y_a=X(K_a)$.

Now, Theorem \ref{SOPIi} yields that
\begin{flalign*}
&d_{TV}\(\frac{F}{\sqrt{\operatorname{Var}F}},N\)^2\leq16\,\int_{\mathbb{R}^2}\sqrt{E\[\(\(D^2F\otimes_1D^2F\)(x,y)\)^2\]E\[\(DF(x)DF(y)\)^2\]}d\mu(x)d\mu(y)
\end{flalign*}
and one has to assess the quantities on the right hand side of the previous inequality. We have
\begin{flalign*}
&E\[\(\(D^2F\otimes_1D^2F\)(x,y)\)^2\]=\\
&\qquad=E\[\(\frac{1}{\mu(A)}\int_{A^2}f''\(Y_a\)f''\(Y_b\)\,\varrho(a-b)K_a(x)K_b(y)\,d\mu(a)d\mu(b)\)^2\]\\
&\qquad\leq\frac{E\abs{f''(N)}^4}{\mu(A)^2}\(\int_{A^2}\abs{\varrho(a-b)K_a(x)K_b(y)}d\mu(a)d\mu(b)\)^2
\end{flalign*}
and, similarly, 
\begin{flalign*}
&E\[\(DF(x)DF(y)\)^2\]
\leq\frac{E\abs{f'(N)}^4}{\mu(A)^2}\(\int_{A^2}\abs{K_a(x)K_b(y)}d\mu(a)d\mu(b)\)^2.
\end{flalign*}
Consequently, we obtain that
\begin{flalign*}
&d_{TV}\(\frac{F}{\sqrt{\operatorname{Var}F}},N\)
\leq\frac{4\(E\abs{f''(N)}^4E\abs{f'(N)}^4\)^{1/4}}{\mu(A)}\,\times\\
&\times\{\int_{A^4}\abs{\varrho(a-b)}\int_{\mathbb{R}}\abs{K_a(x)K_c(x)}d\mu(x)\int_{\mathbb{R}}\abs{K_b(y)K_d(y)}d\mu(y)\,d\mu(a)d\mu(b)d\mu(c)d\mu(d)\}^{1/2}\\
&\leq\frac{c}{\mu(A)}\{\int_{A^2}\abs{\varrho(a-b)}\int_{\mathbb{R}}g(a-x)\(\int_{\mathbb{R}}g(c-x)d\mu(c)\)d\mu(x)\,\times\right.\\
&\qquad\qquad\qquad\qquad\left.\times\int_{\mathbb{R}}g(b-y)\(\int_{\mathbb{R}}g(d-y)d\mu(d)\)d\mu(y)\,d\mu(a)d\mu(b)\}^{1/2}\\
&=\frac{c}{\mu(A)}\{\(\int_{\mathbb{R}}g(w)d\mu(w)\)^4\int_{A^2}\abs{\varrho(a-b)}d\mu(a)d\mu(b)\}^{1/2}\\
&\leq\frac{c}{\mu(A)}\(\int_{\mathbb{R}}g(w)d\mu(w)\)^2\{\mu(A)\int_{\mathbb{R}}\abs{\varrho(x)}d\mu(x)\}^{1/2}=\frac{c}{\sqrt{\mu(A)}}\,,
\end{flalign*}
where 
$$c=4\(E\abs{f''(N)}^4E\abs{f'(N)}^4\)^{1/4}\(\int_{\mathbb{R}}g(w)d\mu(w)\)^2\{\int_{\mathbb{R}}\abs{\varrho(x)}d\mu(x)\}^{1/2}\,,$$ 
which is the desired result.
\endproof

In the next two sections we will see how this result can be applied to more concrete situations.


\subsubsection{Non-linear functionals of continuous stationary Gaussian processes}
In this subsection we apply Proposition \ref{NLFIGP} to more concrete examples and we show how our findings significantly improve the ones in \cite{NPR:09}, see also the discussion in Remark \ref{suboptimal}.

Fix $X$ to be the isonormal Gaussian process generated by the two-sided Brownian motion $\{B_t\}_{t\in \R}$, i.e. 
\begin{equation}\label{IGPBM}
X=\{X(h):=\int_{\R}h(s)\,dB_s\,:\,h\in L^2\(\mathbb{R},\mathscr{B}\(\mathbb{R}\),dx\)\}\,,
\end{equation}
where $B_t=B_1(t)$ when $t\geq0$, $B_t=B_2(-t)$ when $t<0$ and $B_1,B_2$ are two independent standard Brownian motions.
We will apply Proposition \ref{NLFIGP} to three continuous-time models in order to estimate the rate of convergence of some non-linear functionals of continuous-time stationary Gaussian processes towards a Gaussian distribution, all having the following functional form 
\begin{equation}\label{stat_funct}
F_T=\frac{1}{\sqrt{(b-a)T}}\int_{aT}^{bT}\(f\(Y_t\)-E\[f\(N\)\]\)dt\,, \quad a,b\in\R,\, b>a,\, T>0,
\end{equation}
where $Y_t=X\(K_t\)$ for some $K_t\in L^2\(\mathbb{R}_{+},\mathscr{B}\(\mathbb{R}_{+}\),dx\)$, i.e. $Y_t=\int_{\R_+}K_t(s)\,dB_s$.

For the rest of the section, we will assume that $\lim_{T\to\infty} \Var F_T$ exists and it is non-zero and finite. A sufficient condition for its existence is that $f$ is symmetric, see \cite[Proposition 6.3]{NPR:09}.

\subparagraph{The increments of a Brownian motion and the centred Ornstein-Uhlenbeck process.}
The models considered in this paragraph are trivial instances of when Proposition \ref{NLFIGP} holds for $F_T$ as in \eqref{stat_funct}, trivial in the sense that one can easily check that condition \eqref{condition} holds. We start with the case when $Y_t=B_{t+1}-B_{t}\stackrel{law}{=}\int_{\R_+}\1_{[t,t+1)}(s)dB_s$; in this case $K_t=\1_{[t,t+1)}$ and $Y_t$ is stationary, as $\l K_t,K_s\r=1_{[-1,1]}(t-s)=\varrho(t-s)$, with $\int_\R \abs{\varrho(x)}\,dx=2<\infty$. Moreover we have that 
$$
K_t(x)=\1_{[t,t+1)}(x)=\1_{[0,1)}(x-t)=g(x-t) \quad\text{where $g$ is s.t.}\quad \int_\R g(y)\,dy=1<\infty\,.
$$ 

For the second trivial instance one takes $Y_t$ as a centred Ornstein-Uhlenbeck process, namely $Y_t=X\(\sigma\, e^{-\theta(t-x)}\,\1_{(-\infty,t)}(x)\)$, with $\sigma,\theta>0$; this means that $K_t=\sigma\, e^{-\theta(t-x)}\,\1_{(-\infty,t)}(x)$,
$\l K_t,K_s\r=\sigma^2\,e^{-\theta\abs{t-s}}/2\theta=\varrho(t-s)$, i.e. $Y_t$ is stationary, with $\int_\R \abs{\varrho(x)}\,dx=\sigma^2/\theta^2<\infty$. Moreover, we can easily check that 
$$
K_t(x)\leq \sigma\, e^{-\theta(t-x)}\,\1_{(-\infty,1)}(x-t)=g(x-t) \quad\text{where $g$ is s.t.}\quad \int_\R g(y)\,dy<\infty\,.
$$
Thus in both cases condition \eqref{condition} is satisfied and we have that 
$$d_{TV}\(\frac{F_T}{\sqrt{\Var F_T}},N\)\leq \frac{C}{\sqrt{T}}\,,$$
which is a presumably optimal rate for the convergence of $F_T$ to a Gaussian distribution.
\begin{oss}\label{suboptimal}
In \cite[Theorem 6.1]{NPR:09}, the authors obtain a certainly suboptimal rate of convergence for $F_T$, that is
$$
d_{W}\(\frac{F_T}{\sqrt{\Var F_T}},N\)\leq \frac{C}{\,T^{1/4}\,}\,.
$$
This was partly due to the fact that the operator norm of $D^2F_T$ in \eqref{SOPI-NPR} cannot be directly computed, so the authors had to move farther away from the distance in distribution and, using Cauchy-Schwarz inequality, bound $\norm{D^2F_T}_{op}^2$ with $\norm{D^2F_T\otimes_1 D^2F_T}^2_{H^{\otimes 2}}$. 
\end{oss}


\subparagraph{The increments of a fractional Brownian motion.}

We will now show that Proposition \ref{NLFIGP} applies to the case when the process $\{Y_t\}_{t\geq0}$ is defined as the increment of a fractional Brownian motion with Hurst parameter $H<1/2$, that is $Y_t:=B^H_{t+1}-B^H_t$, where $\{B_t^H:t\geq0\}$ is a centred Gaussian process with covariance function $E\[B_t^HB_s^H\]=\frac{1}{2}(t^{2H}+s^{2H}-\abs{t-s}^{2H})$. It is well known that $Y_t$ is stationary and that its correlation function is integrable, see \cite[Proposition 2.2]{Nou:12} and \cite[page 13]{NPR:09}.
The fractional Brownian motion $B_t^H$ has more than one representation in terms of stochastic integral with respect to a two-sided Brownian motion $\{B_t\}_{t\in \R}$, namely in terms of kernels of the isonormal Gaussian process $X$ defined in \eqref{IGPBM}, and we take the following one (see \cite[Section 2.3]{Nou:12})
$$
B^H_t=X\(\underbrace{\frac{1}{c_H}\,\[\(t-u\)^{H-\frac{1}{2}}-\(-u\)^{H-\frac{1}{2}}\1_{(-\infty,0]}(u)+\(t-u\)^{H-\frac{1}{2}}\1_{[0,t)}(u)\]}_{=:\widehat{K}_t(u)}\),
$$
where $c_H$ is a finite constant depending only on $H$.
Hence, thanks to the linearity of $X$, $Y_t=X\(K_t(u)\)$ where $K_t(u)=\widehat{K}_{t+1}(u)-\widehat{K}_t(u)$ and consequently we have that
\begin{flalign*}
\abs{K_t(u)}&=\abs{\widehat{K}_{t+1}(u)-\widehat{K}_t(u)}
=\frac{1}{c_H}\abs{\(t+1-u\)^{H-\frac{1}{2}}\mathbbm{1}_{\{u\in(-\infty,t+1)\}}-\(t-u\)^{H-\frac{1}{2}}\mathbbm{1}_{\{u\in(-\infty,t)\}}}\\
&=\frac{1}{c_H}\abs{\(t-u+1\)^{H-\frac{1}{2}}\mathbbm{1}_{\{(t-u)\in(-1,\infty)\}}-\(t-u\)^{H-\frac{1}{2}}\mathbbm{1}_{\{(t-u)\in(0,\infty)\}}}=:g(t-u)\,.
\end{flalign*}
It remains to prove that $\int g(x)dx<\infty$. We have,
\begin{flalign*}
g(x)&=\frac{1}{c_H}\abs{\(x+1\)^{H-\frac{1}{2}}\mathbbm{1}_{\{x\in(-1,\infty)\}}-x^{H-\frac{1}{2}}\mathbbm{1}_{\{x\in(0,\infty)\}}}\\
&\leq\frac{1}{c_H}\abs{\(x+1\)^{H-\frac{1}{2}}\mathbbm{1}_{\{x\in(-1,0]\}}}+\frac{1}{c_H}\abs{\(x+1\)^{H-\frac{1}{2}}\mathbbm{1}_{\{x\in(0,\infty)\}}-x^{H-\frac{1}{2}}\mathbbm{1}_{\{x\in(0,\infty)\}}}\\
&=:\frac{1}{c_H}\(g_1(x)+g_2(x)\)
\end{flalign*}
Now,
$$
\int_\R g_1(x)dx=\int_{-1}^0\abs{\(x+1\)^{H-\frac{1}{2}}}dx=\frac{1}{H+\frac{1}{2}}<\infty
$$
and
$$\displaylines{
\int_\R g_2(x)dx=\int_{0}^{\infty}\abs{\(x+1\)^{H-\frac{1}{2}}-x^{H-\frac{1}{2}}}dx.
}$$
The function $g_2$ is integrable around $0$ and, for $N$ large enough,
\begin{flalign*}
&\int_{N}^{\infty}\abs{\(x+1\)^{H-\frac{1}{2}}-x^{H-\frac{1}{2}}}dx=\int_{N}^{\infty}\abs{x^{H-\frac{1}{2}}\(\(1+\frac{1}{x}\)^{H-\frac{1}{2}}-1\)}dx\\
&\quad=\int_{N}^{\infty}\abs{x^{H-\frac{3}{2}}\(\frac{\(1+1/x\)^{H-\frac{1}{2}}-1}{1/x}\)}dx\,\sim\int_{N}^{\infty}\abs{x^{H-\frac{3}{2}}\(H-\frac{1}{2}\)}dx<\infty \, ,
\end{flalign*}
for each $H<1/2$. Thus we just proved that $\int g(x)dx<\infty$ and consequently that the increment of a fractional Brownian motion with Hurst parameter $H\in \(0,\frac{1}{2}\)$ satisfies conditions of Proposition \ref{NLFIGP}. This fact leads to the following result which is, to the best of our knowledge, new.
\begin{cor}
Fix $a<b$ in $\mathbb{R}$ and, for any $T>0$, consider the integral functional
$$
F_T=\frac{1}{\sqrt{(a-b)T}}\int_{aT}^{bT}\(f\(B^H_{u+1}-B^H_u\)-E\[f\(N\)\]\)du\,,
$$ 
where $B^H_t$ is a fractional Brownian motion with Hurst parameter $H<1/2$.
Then
$$
d_{TV}\(\frac{F_T}{\sqrt{\Var F_T}},N\)\leq \frac{C}{\sqrt{T}},
$$
where $N\sim\NN(0,1)$ and $C$ is a constant that does not depend on $T$.
\end{cor}

\begin{oss}
\begin{itemize}
\item[(i)] Our result does not guarantee that $\lim_{T\to\infty}\Var F_T$ exists. A sufficient condition to have $\lim_{T\to\infty}\Var F_T\in(0,\infty)$ is that $f$ is symmetric, see \cite[Proposition 6.3]{NPR:09}.
\item [(ii)] Note that when $H=\frac{1}{2}$, $B^H_t$ is a classical Brownian motion and Proposition \ref{NLFIGP} applies. While in the case where $H>\frac{1}{2}$ our result does not apply.
\end{itemize}
\end{oss}


\subsubsection{Non-linear functionals of stationary Gaussian sequences: \\ a Breuer-Major type result}

Proposition \ref{NLFIGP} can be \emph{discretised} to obtain a Breuer-Major type CLT when the Hermite rank of the subordinated Gaussian sequence is greater or equal to $1$\footnote{In general, any CLT involving conditions on Hermite ranks and series of covariance coefficients is usually called a \emph{Breuer-Major Theorem}, in honor of the seminal paper \cite{BM:83}. }. 
Let $X = \{X_k : k \in \Z\}$ be a centered stationary Gaussian sequence with unit variance and such that each $X_k=X(K_k)$, where $X$ is still taken as in \eqref{IGPBM}. For all $\nu \in \Z$, we set $\rho(\nu) = E[X_0X_{\nu}]$ and we assume that
$$\sum_{\nu=-\infty}^{+\infty}\abs{\varrho(\nu)}<\infty\,\,.$$ 
Let
$$
F_n=\frac{1}{\sqrt{n}}\sum_{k=1}^n f\(X_k\)-E\[f\(X_k\)\],
$$
where $f:\mathbb{R}\rightarrow\mathbb{R}$ is a real function of class $C^2$ such that $E\abs{f(N)}<\infty$ and $E\abs{f''(N)}^4<\infty$ when $N\sim\mathcal{N}(0,1)$. We have the following Breuer-Major type result.

\begin{cor}\label{BM} 
Assume that $\abs{K_k(x)}\leq g(k-x)$, 
where $g$ is such that $\int_\R g(y)<\infty$. Then, if $\lim_{n\to\infty}\Var F_n \in(0,\infty)$,
$$
d_{TV}\(\frac{F_n}{\sqrt{\Var F_n}},N\)\leq \frac{C}{\sqrt{n}},
$$
where $N\sim\NN(0,1)$ and $C$ is a constant that does not depend on $n$.
\end{cor}

Hence, as $n\to\infty$, we obtain a quantitative central limit theorem. 

\begin{oss}
The assumptions of Corollary \ref{BM} trivially holds, as in the continuous case showed in the previous section, for both the case of the increment of a Brownian motion, that is $X_k=B_{k+1}-B_{k}$, and the case of a discrete centred Ornstein-Uhlenbeck process, namely 
$$X_k=\gamma X_{k-1}+\sigma \(B_k-B_{k-1}\)\,,$$ 
where $\gamma\in (0,1)$ and $\sigma\in \R_+$.
Indeed, in the latter case, one has that (see \cite{Qin:11})
$$X_k\stackrel{law}{=}X\(\sigma\,\gamma^{k-1-[x]}\,\1_{[0,k)}(x)\)\,.$$
Moreover, Corollary \ref{BM} holds for the increment of a fractional Brownian motion, that is $X_k=B_{k+1}^H-B_{k}^H$, and since the computations are analogous of the ones in the previous section we will not show them here.
\end{oss}


\subsection{Non-linear functionals of a Brownian sheet} \label{sheet}

As a final application in the infinite-dimensional setting, we use our bound in order to estimate the rate of convergence of a non-linear and positive functional of a Brownian sheet towards a standard Gaussian distribution. A particular instance of this model was firstly studied in \cite{PY:04} and then in \cite{NP:05}, where the authors considered a quadratic functional and presented only qualitative central limit theorems. A first quantitative and exact CLT, still just in the case of a quadratic functional, was then presented in \cite{NP:09c}. The rate of convergence obtained therein is exact and as a consequence we will show that also our rate is optimal, as it does not depend on the functional form of the considered model, see Remark \ref{exact}.

A {\it Brownian sheet} $W$ on $[0,1]^n$ is a centred Gaussian process 
$$
W=\{W(x_1,\dots,x_n): (x_1,\dots,x_n)\in[0,1]^n\}
$$
with covariance function $
E\[W(x_1,\dots,x_n)W(y_1,\dots,y_n)\]=\prod_{i=1}^n \(x_i \wedge y_i\)\,.
$
Note that the Gaussian space generated by $W$ can be identified with an isonormal Gaussian process $X$ over $L^2\([0,1]^n,dx_1\cdots dx_n\)$, namely 
$$
W(x_1,\dots,x_n)=\int_{[0,1]^n}\mathbbm{1}_{[0,x_1]}(u_1)\cdots\mathbbm{1}_{[0,x_n]}(u_n)\,dB_{u_1}\cdots dB_{u_n}\,,
$$
where $\{B_t\}_{t\geq0}$ is a standard Brownian motion.
Let $\hat{F}_{\eps}:=(\log1/\eps)^{-n/2}\(F_{\eps}-E\[F_\eps\]\)$, with 
$$
F_{\eps}=\int_{[\eps,1]^n}f\(\frac{W(x_1,\dots,x_n)}{\sqrt{x_1\cdots x_n}}\)d\nu_n(x_1,\dots,x_n)\,,
$$
where $d\nu_n(x_1,\dots,x_n):=\frac{dx_1\cdots dx_n}{x_1\cdots x_n}$, with $d\nu(x):=d\nu_1(x)=\frac{dx}{x}$, and $f:\R\to\R_+$ is a positive function of class $C^2$ such that, for $N\sim\NN\(0,1\)$, $E\[f(N)^2\]<\infty$ and $f$ admits the Hermite expansion $f(x)=\sum_{q=0}^\infty \frac{c_q}{q!}\,H_q(x)$ $\P$-a.s. 

\medskip

\begin{oss}\label{jeu}
First of all, note that
\begin{flalign*}
E\[F_{\eps}\]&=\int_{[\eps,1]^n}E\[f\(\frac{W(x_1,\dots,x_n)}{\sqrt{x_1\cdots x_n}}\)\]d\nu_n(x_1,\dots,x_n)\\
&=\int_{[\eps,1]^n}E\[f\(N\)\]d\nu_n(x_1,\dots,x_n)\\
&=E\[f\(N\)\]\nu_n([\eps,1]^n)=E\[f\(N\)\]\(\log \frac{1}{\eps}\)^n\stackrel{\eps\to0}{\xrightarrow{\hspace*{0.7cm}}}+\infty\,.
\end{flalign*}
Therefore, a modification of Jeulin's Lemma (see \cite[Lemma 1]{Je:80}, as well as \cite{Pe:01}) yields that $\lim_{\eps\to0}F_\eps=+\infty$, $\P$-a.s.
In particular, note that the normalisation constant $\(\log\frac{1}{\eps}\)^{-n/2}$ is chosen in order for $\hat{F}_{\eps}$ to have the variance converging towards a positive finite constant as $\eps$ goes to zero. Indeed, denoting $\mathbf{x}_n:=(x_1,\dots,x_n)$, we have that
\begin{flalign*}
\operatorname{Var}\(F_{\eps}\)&=E(F_{\eps}^2)-\[E(F_{\eps})\]^2\\
&=\int_{[\eps,1]^{2n}}\operatorname{Cov}\bigg(f\(X\(K_{\mathbf{x}}\)\),f\(X\(K_{\mathbf{y}}\)\)\bigg)\,d\nu_n(\mathbf{x}_n)d\nu_n(\mathbf{y}_n)\\
&=\sum_{q=1}^{\infty}\frac{c_q^2}{q!}\(\int_{[\eps,1]^2}\(\frac{x\wedge y}{\sqrt{xy}}\)^q\,d\nu(x)d\nu(y)\)^n=2^n\(\log\frac{1}{\eps}\)^n\sum_{q=1}^{\infty}\frac{c_q^2}{q!}\frac{2^n}{q^n}\,\,,
\end{flalign*}
where $K_{\mathbf{x}}(\mathbf{u})=K_{(x_1,\dots,x_n)}(u_1,\dots,u_n)=\frac{\mathbbm{1}_{[0,x_1]}(u_1)\dots\mathbbm{1}_{[0,x_n]}(u_n)}{\sqrt{x_1\cdots x_n}}$. We finally have to note that, 
$$
E\[f(N)^2\]<\infty\quad\Longrightarrow\quad\sum_{q=1}^{\infty}\frac{c_q^2}{q!}\frac{2^n}{q^n}<\infty \,\,.
$$
\end{oss}
Our result goes as follows.

\medskip

\begin{prop}\label{sheet}
Assume $E\abs{f(N)}^2<\infty$ and $E\abs{f''(N)}^4<\infty$, where $N\sim\mathcal{N}(0,E[\hat{F}_\eps^2])$, then we have that 
$$
d_{TV}\(\hat{F}_{\eps},N\)\leq \frac{C_n}{\(\log\frac{1}{\eps}\)^{n/2}},
$$
where $C_n$ is a constant that does not depend on $\eps$.
\end{prop}

\begin{oss}\label{exact}
In \cite[Proposition 5.2]{NP:09c}, the authors obtain an exact rate of convergence in Kolmogorov distance only in the case $f(x)=x^2$, that is
\begin{equation}\label{exact_sheet}
\frac{c_n}{\(\log\frac{1}{\eps}\)^{n/2}}\leq d_{Kol}\(\hat{F}_{\eps},N\)\leq \frac{C_n}{\(\log\frac{1}{\eps}\)^{n/2}}\,,
\end{equation}
where, again, $c_n$ and $C_n$ are constants that do not depend on $\eps$. This proves that, since the exact rate in \eqref{exact_sheet} does not depend on the form of $f$, also our generalisation of Proposition \ref{sheet} attains an optimal rate of convergence for $\hat{F}_{\eps}$.
\end{oss}

\proof
We can write $\hat{F}_{\eps}$ as follows
\begin{eqnarray*}
\hat{F}_{\eps}&=&\frac{1}{\(\log\frac{1}{\eps}\)^{n/2}}\int_{[\eps,1]^n}\{f\(\frac{W(x_1,\dots,x_n)}{\sqrt{x_1\cdots x_n}}\)-E\[f\(N\)\]\}d\nu_n(x_1,\dots,x_n)\\
&=&\frac{1}{\(\log\frac{1}{\eps}\)^{n/2}}\int_{[\eps,1]^n}\{f\(X\(K_{\mathbf{x}}\)\)-E\[f\(N\)\]\}d\nu_n(\mathbf{x}_n)\,.
\end{eqnarray*}
As a consequence, thanks to the stochastic Fubini theorem (see \cite{Ve:12}), we can compute
\begin{eqnarray}
D\hat{F}_{\eps}(\mathbf{t})&=&\frac{1}{\(\log\frac{1}{\eps}\)^{n/2}}\int_{[\eps,1]^n}f'\(X\(K_{\mathbf{x}}\)\)K_{\mathbf{x}}(\mathbf{t})d\nu_n(\mathbf{x}_n)\label{D}
\end{eqnarray}
and
\begin{eqnarray}
D^2\hat{F}_{\eps}(\mathbf{t},\mathbf{s})&=&\frac{1}{\(\log\frac{1}{\eps}\)^{n/2}}\int_{[\eps,1]^n}f''\(X\(K_{\mathbf{x}}\)\)\,K_{\mathbf{x}}(\mathbf{t}) \, K_{\mathbf{x}}(\mathbf{s}) \, d\nu_n(\mathbf{x}_n)\,.\label{D2}
\end{eqnarray}
So we have that
\begin{flalign*}
&E\[\(D\hat{F}_{\eps}(\mathbf{t})D\hat{F}_{\eps}(\mathbf{s})\)^2\]=\\
&=\frac{1}{\(\log\frac{1}{\eps}\)^{2n}}\, E\[\(\int_{[\eps,1]^{2n}}f'\(X\(K_{\mathbf{x}}\)\)f'\(X\(K_{\mathbf{y}}\)\)K_{\mathbf{x}}(\mathbf{t})K_{\mathbf{y}}(\mathbf{s})d\nu_n(\mathbf{x}_n)d\nu_n(\mathbf{y}_n)\)^2\]\\
&\leq \frac{1}{\(\log\frac{1}{\eps}\)^{2n}}E\abs{f'(N)}^4\prod_{i=1}^n\(\int_{[\eps,1]^{2}}\frac{\mathbbm{1}_{[0,x_i]}(t_i)}{\sqrt{x_i}}\frac{\mathbbm{1}_{[0,w_i]}(t_i)}{\sqrt{w_i}}d\nu(x_i)d\nu(w_i)\)\times\\
&\qquad\qquad\qquad\qquad\qquad\qquad\times\(\int_{[\eps,1]^{2}}\frac{\mathbbm{1}_{[0,y_i]}(s_i)}{\sqrt{y_i}}\frac{\mathbbm{1}_{[0,z_i]}(s_i)}{\sqrt{z_i}}d\nu(y_i)d\nu(z_i)\)\\
&=\frac{1}{\(\log\frac{1}{\eps}\)^{2n}}E\abs{f'(N)}^4\prod_{i=1}^n\(\int_{t_i\vee\eps}^1\,\frac{1}{x_i^{3/2}}\,dx_i\)^2\(\int_{s_i\vee\eps}^1\,\frac{1}{y_i^{3/2}}\,dy_i\)^2\\
&=\frac{1}{\(\log\frac{1}{\eps}\)^{2n}}E\abs{f'(N)}^4\(2^{4n}\prod_{i=1}^n\(\frac{1}{\sqrt{t_i\vee\eps}}-1\)^2\(\frac{1}{\sqrt{s_i\vee\eps}}-1\)^2\).
\end{flalign*}
Now, without loss of generality, consider the part of the space $[0,1]^n$ in which $x_i\leq y_i$, for every $i$:
\begin{flalign*}
&\(D^2\hat{F}_{\eps}\otimes_1D^2\hat{F}_{\eps}\)(\mathbf{t},\mathbf{s})=\frac{1}{\(\log\frac{1}{\eps}\)^{n}}\int_{[\eps,1]^{2n}}f''\(X\(K_{\mathbf{x}}\)\)f''\(X\(K_{\mathbf{y}}\)\)\times\\
&\qquad\qquad\times\(\int_{[0,1]^n}K_{\mathbf{x}}(\mathbf{u})K_{\mathbf{y}}(\mathbf{u})du_1\cdots du_n\)\,K_{\mathbf{x}}(\mathbf{t})K_{\mathbf{y}}(\mathbf{s})\,d\nu_n(\mathbf{x}_n)d\nu_n(\mathbf{y}_n)\\
&=\frac{1}{\(\log\frac{1}{\eps}\)^{n}}\int_{[\eps,1]^{2n}}f''\(X\(K_{\mathbf{x}}\)\)f''\(X\(K_{\mathbf{y}}\)\)E[X\(K_{\mathbf{x}}\)X\(K_{\mathbf{y}}\)]K_{\mathbf{x}}(\mathbf{t}) \,K_{\mathbf{y}}(\mathbf{s})\,d\nu_n(\mathbf{x}_n)d\nu_n(\mathbf{y}_n)\\
&=\frac{1}{\(\log\frac{1}{\eps}\)^{n}}\int_{[\eps,1]^{2n}}f''\(X\(K_{\mathbf{x}}\)\)f''\(X\(K_{\mathbf{y}}\)\)\,\prod_{i=1}^n \,\frac{\mathbbm{1}_{[0,x_i]}(t_i)}{x_i}\frac{\mathbbm{1}_{[0,y_i]}(s_i)}{y_i^2}\,dx_idy_i
\end{flalign*}
so that (note that since we are treating the case $x_i\leq y_i$, for every $i$, this implies that $t_i\leq s_i$, for every $i$)
\begin{flalign*}
&E\[\(D^2\hat{F}_{\eps}\otimes_1D^2\hat{F}_{\eps}\)(\mathbf{t},\mathbf{s})^2\]=\\
&=\frac{1}{\(\log\frac{1}{\eps}\)^{2n}}E\[\int_{[\eps,1]^{4n}}f''\(X\(K_{\mathbf{x}}\)\)f''\(X\(K_{\mathbf{y}}\)\)f''\(X\(K_{\mathbf{w}}\)\)f''\(X\(K_{\mathbf{z}}\)\)\times\right.\\
&\qquad\qquad\qquad\qquad\times \,\prod_{i=1}^n \,\frac{\mathbbm{1}_{[0,x_i]}(t_i)}{x_i}\frac{\mathbbm{1}_{[0,y_i]}(s_i)}{y_i^2}\,dx_idy_i \prod_{i=1}^n \,\frac{\mathbbm{1}_{[0,w_i]}(t_i)}{w_i}\frac{\mathbbm{1}_{[0,z_i]}(s_i)}{z_i^2}\,dw_idz_i\bigg]\\
&\leq\frac{1}{\(\log\frac{1}{\eps}\)^{2n}}E\abs{f''(N)}^4\,\prod_{i=1}^n\(\int_{[\eps,1]^{2}}\,\frac{\mathbbm{1}_{[0,x_i]}(t_i)}{x_i}\frac{\mathbbm{1}_{[0,y_i]}(s_i)}{y^2}\,dx_idy_i\)^{2}\\
&=\frac{1}{\(\log\frac{1}{\eps}\)^{2n}}E\abs{f''(N)}^4\prod_{i=1}^n\(2\int_{s_i\vee \eps}^1\,\frac{dy_i}{y_i^2} \, \int_{t_i\vee \eps}^{y_i} \,\frac{dx_i}{x_i}\mathbbm{1}_{\{t_i\leq s_i\}}+2\int_{t_i\vee \eps}^1\,\frac{dy_i}{y_i^2} \, \int_{t_i\vee \eps}^{y_i} \,\frac{dx_i}{x_i}\mathbbm{1}_{\{s_i\leq t_i\}}\)^{2}\\
&=\frac{1}{\(\log\frac{1}{\eps}\)^{2n}}E\abs{f''(N)}^42^{4n}\prod_{i=1}^n\(\[-\frac{\log x_i}{x_i}-\frac{1}{x_i}-\frac{\log (t_i\vee \eps)}{x_i}\]_{s_i\vee \eps}^1\)^{2}\\
&=\frac{E\abs{f''(N)}^42^{4n}}{\(\log\frac{1}{\eps}\)^{2n}}\prod_{i=1}^n\(\[-\log (t_i\vee \eps)+\frac{\log (s_i\vee \eps)}{s_i\vee \eps}+\frac{1}{s_i\vee \eps}+\frac{\log (t_i\vee \eps)}{s_i\vee \eps}-1\]\)^{2}\,.
\end{flalign*}
Let $Z$ be a Gaussian random variable with same mean and variance of $\hat{F}_{\eps}$, then, plugging into our bound the previous quantities, we have that 
\begin{flalign*}
d\(\hat{F}_{\eps},Z\)^2 &\leq\,\frac{16}{\(\Var\hat{F}_{\eps}\)^2}\frac{1}{\(\log\frac{1}{\eps}\)^{2n}}E\abs{f'(N)}^2E\abs{f''(N)}^22^{4n}\\
&\qquad\times\int_{[0,1]^{2n}}\,dt_1\cdots dt_n\,ds_1\cdots ds_n\prod_{i=1}^n\(\frac{1}{\sqrt{t_i\vee\eps}}-1\)\(\frac{1}{\sqrt{s_i\vee\eps}}-1\)\\
&\qquad\times \(\[-\log (t_i\vee \eps)+\frac{\log (s_i\vee \eps)}{s_i\vee \eps}+\frac{1}{s_i\vee \eps}+\frac{\log (t_i\vee \eps)}{s_i\vee \eps}-1\]\)\\ 
& \stackrel{\eps\to0}{\approx} \frac{16 \,\(\sum_{q=1}^{\infty}\frac{c_q^2}{q!}\frac{2}{q}\)^{-2}\, 2^{-2n}}{\(\log\frac{1}{\eps}\)^{2n}}E\abs{f'(N)}^2E\abs{f''(N)}^22^{4n}\\
&\qquad\times \(\int_{\eps}^1\,\frac{ds}{\sqrt{s}}\int_{\eps}^s\,\frac{dt}{\sqrt{t}}\(-\log (t)+\frac{\log (s)}{s}+\frac{1}{s}+\frac{\log (t)}{s}\)\)^n\\
&\stackrel{\eps\to0}{\approx} \frac{16 \, c}{\(\log\frac{1}{\eps}\)^{2n}}E\abs{f'(N)}^2E\abs{f''(N)}^22^{2n}\\
&\quad\times \(\int_{\eps}^1\,\frac{ds}{\sqrt{s}}\(-2\sqrt{s}(\log (s)-2)+2\sqrt{s}\frac{\log (s)}{s}+\frac{2\sqrt{s}}{s}+\frac{2\sqrt{s}(\log (s)-2)}{s}\)\)^n\\
&\stackrel{\eps\to0}{\approx} \frac{16\, c}{\(\log\frac{1}{\eps}\)^{2n}}E\abs{f'(N)}^2E\abs{f''(N)}^22^{2n}\,\(\int_{\eps}^1\,\frac{2}{s}\,ds\)^n\\
&= \frac{16\,c}{\(\log\frac{1}{\eps}\)^{n}}E\abs{f'(N)}^2E\abs{f''(N)}^22^{3n}\,,
\end{flalign*}
where $1/c=\(\sum_{q=1}^{\infty}\frac{c_q^2}{q!}\frac{2^n}{q^n}\)^2<\infty$ (see Remark \ref{jeu}).

The other cases, i.e. the other parts of the space $[0,1]^n$, can be treated analogously since all the functions considered here are symmetric.
\endproof


\section{An application to random matrices: Traces of Wigner matrices}\label{Wigner}

Let $X=\(X_{ij}\)_{1\leq i\leq j\leq n}$ be a vector with values in $\mathbb{R}^{\frac{n(n+1)}{2}}$ and $Y\(X\)=\(Y_{ij}(X)\)_{1\leq i, j\leq n}$ be the $n\times n$ matrix whose $ij$-th entry is $X_{ij}$ if $i\leq j$ and $X_{ji}$ if $i>j$. The random matrix  
$$
A\(X\)=\frac{1}{\sqrt{n}}Y\(X\),\quad n\geq1,
$$
is called {\it Wigner matrix} of dimension $n\times n$. In the famous paper \cite{SS:98}, the authors described the limiting behaviour of $\Tr\(A\(X\)^{p_n}\)$ when $p_n=o(n^{2/3})$, obtaining a qualitative (i.e. non-quantitative) CLT for $\Tr\(A\(X\)^{p_n}\)$, when the $X_{ij}$'s are independent centred symmetric random variables such that $E\[X_{ij}^2\]=\frac{1}{4}$ and their higher moments do not grow faster than the moments of a Gaussian random variable. Their main result even implied CLTs for more general class of linear statistics of the eigenvalues of $A(X)$ as well as almost sure convergence of the greatest eigenvalue of $A(X)$ to $1$ (interested readers can see \cite[Corollary 1, 2]{SS:98}). Later on, Chatterjee \cite{Ch:09} applied his formulation of second order Poincar\'e inequality (see Theorem \ref{ch_thm}) to obtain a QCLT in the case when the $X_{ij}$'s are both independent centred symmetric random variables such that $c\leq E\[X_{ij}^2\]\leq C$ and twice differentiable functions, with bounded first and second derivatives, of a standard Gaussian random variable; however assuming $p_n=o(\log n)$. In this section, we consider the case when $X\sim \frac{1}{2}\times\mathcal{N}\(0,I_{d\times d}\)$, $d=\frac{n(n+1)}{2}$ and we obtain a quantitative CLT in total variation distance for $\(\Tr(A^{p_n})-E\[\Tr(A^{p_n})\]\)/\sqrt{\Var\Tr A^{p_n}}$ when $p_n=o(n^{4/15})$. We stress that one could use our form of second order Poincar\'e inequalities to obtain a QCLT in the more general case considered in \cite{Ch:09} but instead allowing $p_n=o(n^{4/15})$.

\subsection{Main result}
From now on, for sake of notational simplicity, we will write $p$ intended as $p_n$. Our main result is the following.

\begin{thm}\label{WignerTHM}
If $p=o(n^{4/15})$, then $F_n:=\frac{\Tr A(X)^p-E[\Tr A(X)^p]}{\sqrt{\Var \Tr A(X)^p}}\to N\sim\NN\(0,1\)$ in distribution as $n\to \infty$ and there exists a universal constant $C<\infty$ such that
$$
d_{TV}(F_n,N)\leq \,C\,\(\frac{e^{3/4}}{(2\,\pi)^{3/8}} \,\frac{p^{7/8}}{n^{1/4}} + \frac{2\,e}{2^{1/8}\,\sqrt{\pi}} \,\frac{p^{15/8}}{\sqrt{\,n\,}}\)\,\,.
$$

Moreover, setting $p=O(n^{\alpha})$, we have that
\begin{equation}\label{cases-Wig}
d_{TV}(F_n,N)= \begin{cases} 
O\(\frac{p^{15/8}}{\sqrt{\,n\,}}\) &\mbox{if } \quad \frac14 \leq \alpha < \frac 4{15}\\[10pt]
O\(\frac{p^{7/8}}{n^{1/4}}\) &\mbox{if } \quad 0 < \alpha < \frac14 
\end{cases}
\end{equation}
\end{thm}

\begin{oss}
Our proof shows that there exists a numerical sequence $\eta_n\in(0,\infty)$ such that for every $n$
$$
d_{TV}(F_n,N)\leq \,\eta_n\,\(\frac{e^{3/4}}{(2\,\pi)^{3/8}} \,\frac{p^{7/8}}{n^{1/4}} + \frac{2\,e}{2^{1/8}\,\sqrt{\pi}} \,\frac{p^{15/8}}{\sqrt{\,n\,}}\)\,\,,
$$
and $\eta_n\to4\pi$ as $n\to\infty$.
\end{oss}

\begin{oss}\label{comparison}
Assumptions of Theorem \ref{WignerTHM} can be seen as a special instance of the model considered in \cite{Ch:09}, namely here we take $c=C=\frac{1}{4}$ and we take the $X_{ij}$'s to be Gaussian themselves. In such case, the findings from \cite{Ch:09} would anyway led to a QCLT only when $p=o(\log n)$. In this sense, our result can be seen as an improvement in terms of speed of $p$.
However, our result cannot achieve the level of generality of \cite{SS:98,SS:99}, where not only a qualitative CLT is reached for $p=o(n^{2/3})$ but the authors do not assume any Gaussianity (not even subordinated, as in \cite{Ch:09}). In general, in order to apply second order Poincar\'e inequalities, some subordinated Gaussianity is needed.
\end{oss}

\subsection{First computations and sketch of the proof}

Note that the following relations hold (see Lemma 5.4 in \cite{Ch:09}):
\begin{equation*}
\frac{\partial}{\partial a_{ij}}\operatorname{Tr}\(A^p\)=p\(A^{p-1}\)_{ji}  
\end{equation*}
and
\begin{flalign*}
\frac{\partial^2}{\partial a_{ij}\partial a_{rs}}\operatorname{Tr}(A^p)&=p\sum_{q=0}^{p-2}\operatorname{Tr}\(\frac{\partial A}{\partial a_{ij}}A^q\frac{\partial A}{\partial a_{rs}}A^{p-2-q}\)&\\
&=p\sum_{q=0}^{p-2}\operatorname{Tr}\(E_{ij}A^qE_{rs}A^{p-2-q}\)=p\sum_{q=0}^{p-2}\(A^q\)_{jr}\(A^{p-2-q}\)_{is},
\end{flalign*}
where $E_{ij}$ is the $n\times n$ matrix whose entries are all zero except for the $ij$-th.
Now, note that we can write $X=\frac{1}{2}\,Z$, where $Z\sim\NN\(0,I_{d\times d}\)$. Then, $g(x)=g(\frac{z}{2})=f(z)$, and we have
\begin{flalign*}
\frac{\partial f}{\partial z_{kl}}(z)&=\sum_{i,j=1}^n\frac{\partial}{\partial a_{ij}}\operatorname{Tr}(A^p)\frac{\partial a_{ij}}{\partial z_{kl}}(z)\\
&=\sum_{i,j=1}^n p\(A^{p-1}\)_{ji}\frac{1}{2\,\sqrt{n}}\(\mathbbm{1}_{\{(k,l)=(i,j)\}}+\mathbbm{1}_{\{(k,l)=(j,i)\}}\mathbbm{1}_{\{k\neq l\}}\)\\
&=\frac{p}{2\,\sqrt{n}}\(A^{p-1}\)_{kl}+\frac{p}{2\,\sqrt{n}}\(A^{p-1}\)_{kl}\mathbbm{1}_{\{k\neq l\}}
\end{flalign*}
and
\begin{flalign*}
&\frac{\partial^2 f}{\partial z_{kl}\partial z_{hm}}(z)=\sum_{i,j=1}^n\frac{\partial}{\partial a_{ij}}\operatorname{Tr}(A^p)\frac{\partial^2 a_{ij}}{\partial x_{kl}\partial x_{hm}}+\sum_{i,j,r,s=1}^n\frac{\partial^2}{\partial a_{ij}\partial a_{rs}}\operatorname{Tr}(A^p)\frac{\partial a_{ij}}{\partial x_{kl}}\frac{\partial a_{rs}}{\partial x_{hm}}\\
&=\sum_{i,j,r,s=1}^np\sum_{q=0}^{p-2}\{\(A^q\)_{jr}\(A^{p-2-q}\)_{is}\}\frac{\partial a_{ij}}{\partial x_{kl}}\frac{\partial a_{rs}}{\partial x_{hm}}\\
&=\sum_{i,j,r,s=1}^n p\sum_{q=0}^{p-2}\(A^q\)_{jr}\(A^{p-2-q}\)_{is}\frac{1}{4\,n}\(\mathbbm{1}_{\{(k,l)=(i,j)\}}+\mathbbm{1}_{\{(k,l)=(j,i)\}}\mathbbm{1}_{\{k\neq l\}}\)\\
&\qquad\qquad\qquad\qquad\qquad\qquad\qquad\qquad\(\mathbbm{1}_{\{(h,m)=(r,s)\}}+\mathbbm{1}_{\{(h,m)=(s,r)\}}\mathbbm{1}_{\{h\neq m\}}\)\\
&=\frac{p}{4\,n}\sum_{q=0}^{p-2}\{\(A^q\)_{lh}\(A^{p-2-q}\)_{mk}+\(A^q\)_{lm}\(A^{p-2-q}\)_{hk}\mathbbm{1}_{\{h\neq m\}}+\qquad\qquad\right.\\
&\qquad\qquad\left.+\(A^q\)_{kh}\(A^{p-2-q}\)_{ml}\mathbbm{1}_{\{k\neq l\}}+\(A^q\)_{km}\(A^{p-2-q}\)_{hl}\mathbbm{1}_{\{k\neq l\}}\mathbbm{1}_{\{h\neq m\}}\}
\end{flalign*}
Plugging these relations into \eqref{SOPfTV} we deduce that
\begin{flalign}\label{generale}
&d_{TV}(F_n,N)^2=d_{TV}\(G_n-E[G_n]\,,\,\NN\(0,\Var G_n\)\)^2\nonumber\\
&\leq\frac{16}{\(\Var G_n\)^2}\sum_{i,k,l,m=1}^n \{E\[\(\sum_{j,h=1}^n\nabla^2_{ik,jh}g\nabla^2_{lm,jh}g\)^2\]\}^{1/2}\{E\[\(\nabla_{ik}g\nabla_{lm}g\)^2\]\}^{1/2}\\
&=\frac{1}{\(\Var G_n\)^2}\sum_{i,k,l,m=1}^n\{ E\[ \vast( \sum_{j,h=1}^n\Bigg(\frac{p}{n}\sum_{q=0}^{p-2}\{\(A^q\)_{kj}\(A^{p-2-q}\)_{hi}+\(A^q\)_{kh}\(A^{p-2-q}\)_{ji}\mathbbm{1}_{\{j\neq h\}}+\qquad\qquad\right.\right.\right.\\
&\qquad\qquad\left.+\(A^q\)_{ij}\(A^{p-2-q}\)_{hk}\mathbbm{1}_{\{k\neq i\}}+\(A^q\)_{ih}\(A^{p-2-q}\)_{jk}\mathbbm{1}_{\{k\neq i\}}\mathbbm{1}_{\{j\neq h\}}\}\Bigg)\times\nonumber\\
&\qquad\qquad\qquad\times\,\Bigg(\frac{p}{n}\sum_{q=0}^{p-2}\{\(A^q\)_{mj}\(A^{p-2-q}\)_{hl}+\(A^q\)_{mh}\(A^{p-2-q}\)_{jl}\1_{\{j\neq h\}}+\right.\nonumber\\
&\qquad\qquad\qquad\left.\left.\left.+\(A^q\)_{lj}\(A^{p-2-q}\)_{hm}\1_{\{l\neq m\}}+\(A^q\)_{lh}\(A^{p-2-q}\)_{jm}\1_{\{j\neq h\}}\1_{\{l\neq m\}}\}\Bigg)\vast)^2\,\]\}^{1/2}\times\nonumber\\
&\qquad\qquad\times \{E\[\(\frac{p}{2\,\sqrt{n}}\(A^{p-1}\)_{ik}+\frac{p}{2\,\sqrt{n}}\(A^{p-1}\)_{ik}\1_{\{k\neq i\}}\)^2\,\times\right.\right.\notag\\
&\left.\left.\qquad\qquad\qquad\qquad\qquad\times \(\frac{p}{2\,\sqrt{n}}\(A^{p-1}\)_{lm}+\frac{p}{2\,\sqrt{n}}\(A^{p-1}\)_{lm}\1_{\{l\neq m\}}\)^2\]\}^{1/2}\,.
\end{flalign}
Define
\begin{flalign*}
B_{iklm}&:=\sum_{j,h=1}^n\Bigg(\frac{p}{n}\sum_{q=0}^{p-2}\{\(A^q\)_{kj}\(A^{p-2-q}\)_{hi}+\(A^q\)_{kh}\(A^{p-2-q}\)_{ji}\mathbbm{1}_{\{j\neq h\}}+\qquad\qquad\right.\\
&\qquad\qquad\left.+\(A^q\)_{ij}\(A^{p-2-q}\)_{hk}\mathbbm{1}_{\{k\neq i\}}+\(A^q\)_{ih}\(A^{p-2-q}\)_{jk}\mathbbm{1}_{\{k\neq i\}}\mathbbm{1}_{\{j\neq h\}}\}\Bigg)\times\\
&\qquad\qquad\qquad\times\,\Bigg(\frac{p}{n}\sum_{q=0}^{p-2}\{\(A^q\)_{mj}\(A^{p-2-q}\)_{hl}+\(A^q\)_{mh}\(A^{p-2-q}\)_{jl}\1_{\{j\neq h\}}+\right.\\
&\qquad\qquad\qquad\left.+\(A^q\)_{lj}\(A^{p-2-q}\)_{hm}\1_{\{l\neq m\}}+\(A^q\)_{lh}\(A^{p-2-q}\)_{jm}\1_{\{j\neq h\}}\1_{\{l\neq m\}}\}\Bigg)\,.
\end{flalign*}

Now, without losing any generality (see Remark \ref{wlg}), we can assume that $i\neq k$, $l\neq m$ and $j\neq h$ and consequently $B_{iklm}$ becomes
\begin{flalign*}
B_{iklm}&=\sum_{j,h=1}^n\(\frac{p}{n}\sum_{q=0}^{p-2}\{\(A^q\)_{kj}\(A^{p-2-q}\)_{hi}+\(A^q\)_{kh}\(A^{p-2-q}\)_{ji}+\right.\right.\\
&\qquad\qquad\qquad\left.+\(A^q\)_{ij}\(A^{p-2-q}\)_{hk}+\(A^q\)_{ih}\(A^{p-2-q}\)_{jk}\}\Bigg)\times\\
&\times\,\(\frac{p}{n}\sum_{q=0}^{p-2}\{\(A^q\)_{mj}\(A^{p-2-q}\)_{hl}+\(A^q\)_{mh}\(A^{p-2-q}\)_{jl}+\right.\right.\\
&\qquad\qquad\qquad\left.+\(A^q\)_{lj}\(A^{p-2-q}\)_{hm}+\(A^q\)_{lh}\(A^{p-2-q}\)_{jm}\}\Bigg).
\end{flalign*}
We have that
\begin{flalign}\label{Biklm}
B_{iklm}&=\frac{2\,p^2}{n^2}\sum_{q_1,q_2=0}^{p-2}\{\(A^{q_1+q_2}\)_{km}\(A^{2p-4-q_1-q_2}\)_{il}+\right.\qquad\qquad\qquad\qquad\qquad\qquad\qquad\qquad\notag\\
&+\(A^{q_1+p-2-q_2}\)_{kl}\(A^{q_2+p-2-q_1}\)_{im}+\(A^{q_1+q_2}\)_{kl}\(A^{2p-4-q_1-q_2}\)_{im}+\notag\\
&+\(A^{q_1+p-2-q_2}\)_{km}\(A^{q_2+p-2-q_1}\)_{il}+\(A^{q_1+q_2}\)_{im}\(A^{2p-4-q_1-q_2}\)_{kl}+\notag\\
&+\(A^{q_1+p-2-q_2}\)_{il}\(A^{q_2+p-2-q_1}\)_{km}+\(A^{q_1+q_2}\)_{il}\(A^{2p-4-q_1-q_2}\)_{km}+\notag\\
&\qquad\qquad\qquad\qquad\qquad\qquad\qquad\qquad\left.+\(A^{q_1+p-2-q_2}\)_{im}\(A^{q_2+p-2-q_1}\)_{kl}\}\notag\\
&=\frac{16\,p^2}{n^2}\sum_{q_1,q_2=0}^{p-2}\(A^{q_1+q_2}\)_{km}\(A^{2p-4-q_1-q_2}\)_{il}\notag\\
&=\frac{16\,p^2}{n^2}\sum_{Q_1=0}^{2p-4}\(Q_1+1\)\(A^{Q_1}\)_{km}\(A^{2p-4-Q_1}\)_{il}\,,
\end{flalign}
where $Q_1=q_1+q_2$.
Hence
\begin{flalign*}
B_{iklm}^2&=\frac{2^8\,p^4}{n^4}\sum_{Q_1,Q_2=0}^{2p-4}\(Q_1+1\)\(Q_2+1\)\(A^{Q_1}\)_{km}\(A^{2p-4-Q_1}\)_{il}\(A^{Q_2}\)_{km}\(A^{2p-4-Q_2}\)_{il}\,.
\end{flalign*}
It follows that the first term of the product in \eqref{generale} has the form 
\begin{flalign}\label{defA1}
&\A_1(i,k,l,m):=E\[B_{iklm}^2\]\notag\\
&=\frac{2^8\,p^4}{n^4}\sum_{Q_1,Q_2=0}^{2p-4}\(Q_1+1\)\(Q_2+1\)E\bigg[\(A^{Q_1}\)_{km}\(A^{2p-4-Q_1}\)_{il}\(A^{Q_2}\)_{km}\(A^{2p-4-Q_2}\)_{il}\bigg]\,.
\end{flalign}
Hence, considering the fact that $\Var G_n\to\frac{1}{\pi}$ as $n\to\infty$ (see Theorem A.2 in the auxiliary file Appendix A: {\footnotesize\url{https://annavidotto.files.wordpress.com/2018/06/auxiliary_file-appendix.pdf}}), estimate \eqref{generale} becomes 
\begin{equation}\label{final}
d_{TV}(F_n,N)^2\leq\,\pi^2\, \sum_{i,k,l,m=1}^n\(\A_1(i,k,l,m)\)^{1/2}\(\A_2(i,k,l,m)\)^{1/2},
\end{equation}
where 
\begin{equation}\label{defA2}
\A_2(i,k,l,m):=\frac{p^4}{n^2}E\[\(A^{p-1}\)_{ik}\(A^{p-1}\)_{lm}\(A^{p-1}\)_{ik}\(A^{p-1}\)_{lm}\].
\end{equation}

\begin{oss}\label{wlg}
It is important to note that the other cases, for instance the case where $i=k$, $l\neq m$ and $j\neq h$, 
give exactly the same bound as in \eqref{final}, except for the fact that in \eqref{defA1} and \eqref{defA2} the constants that are independent of $n$ and $p$, $2^8$ and $1$ respectively, will change accordingly. Indeed, when $i=k$, $l\neq m$ and $j\neq h$, \eqref{defA1} and \eqref{defA2} become
\begin{flalign*}
&\A_1(i,k,l,m)\\
&=\frac{2^6\,p^4}{n^4}\sum_{Q_1,Q_2=0}^{2p-4}\(Q_1+1\)\(Q_2+1\)E\bigg[\(A^{Q_1}\)_{km}\(A^{2p-4-Q_1}\)_{il}\(A^{Q_2}\)_{km}\(A^{2p-4-Q_2}\)_{il}\bigg]
\end{flalign*}
and
\begin{equation*}
\A_2(i,k,l,m)=\frac{p^4}{4\,n^2}E\[\(A^{p-1}\)_{ik}\(A^{p-1}\)_{lm}\(A^{p-1}\)_{ik}\(A^{p-1}\)_{lm}\]
\end{equation*}
respectively. In general, for \eqref{defA1}, the constant will be a power between $1$ and $2^8$; while for \eqref{defA2}, it will be a fraction between $\frac{1}{2^4}$ and $1$. However, as the reader will see, the cases in which some index $i,k,l,m,j,h$ is assumed to be equal to another can not give the main contribution to the bound \eqref{final}. For this reason, we will keep the constants associated with the case $i\neq k$, $l\neq m$ and $j\neq h$, without affecting the forthcoming results.
\end{oss}

\begin{notation} 
Given $p=p_n$ such that $\lim_{n\to\infty}p_n=\infty$, and sequences \\$\{A(p,n),C(p,n):n\geq1\}$ such that $A(p,n)$ possibly depends on indices \\$i,k,l,m,Q_1,Q_2$, we will write 
$$
A(p,n)=o\(C(p,n)\) 
$$
to indicate the relation
$$
\frac{A(p,n)}{C(p,n)}\leq\eps_n,
$$
where $\eps_n \to 0$ as $n\to\infty$ and $\eps_n$ does not depend on $i,k,l,m,Q_1,Q_2$. 
\end{notation}

For the rest of the paper, assume that $p_n=o(n^{4/15})$.

The (quite technical) proofs of the forthcoming Propositions \ref{A1} and \ref{A2} are presented in detail in Appendix A, that the reader can find at the following link: \url{https://annavidotto.files.wordpress.com/2018/06/auxiliary_file-appendix.pdf}. 

\bigskip

\begin{prop}\label{A1}
For fixed $i,k,l,m$, we have that
\begin{flalign*}
&\A_1(i,k,l,m)\leq\frac{4^4p^4}{n^4}\sum_{Q_1,Q_2=0}^{2p-4}\,\(Q_1+1\)\(Q_2+1\)\times\\
&\times\{\[\frac{e}{2\,\sqrt{2\,\pi\,p^3}}\mathbbm{1}_{\{i=l\}}+\frac{e}{2\,n\,\sqrt{2\,\pi\,p^3}}\mathbbm{1}_{\{i\neq l\}}\]\mathbbm{1}_{\{Q_1=Q_2=0\}}+\right.\\
&\qquad\quad+2\,\[\frac{e^2}{\sqrt{2\,\pi^2\,p^3\, Q^3}}\mathbbm{1}_{\{i= l,k=m\}}+\frac{2\,e^2}{n\,\sqrt{2\,\pi^2\,p^3\, Q^3}}\mathbbm{1}_{\{i\neq l,k=m\}}+\right.\\
&\qquad\qquad\qquad\left.+\frac{e^2}{n^2\,\sqrt{2\,\pi^2\,p^3\, Q^3}}\mathbbm{1}_{\{i\neq l,k\neq m\}}\]\mathbbm{1}_{\{Q_1, Q_2 \text{ even, } Q_1\neq0\}}+\\
&\left.\qquad\qquad\qquad\qquad\qquad\qquad+\[\frac{e^2}{n^2\,\sqrt{2\,\pi^2\,p^3\, Q^3}}\]\mathbbm{1}_{\{Q_1, Q_2 \text{ odd}\}}\}\,(1+o(1))\,,
\end{flalign*}
where $2\,Q=Q_1+Q_2$ and $o(1)$ indicates a numerical sequence converging to zero, as $n\uparrow\infty$.
\end{prop}

\begin{prop}\label{A2}
For fixed $i,k,l,m$, we have that
\begin{flalign}
&\A_2(i,k,l,m)\leq\notag\\
&\leq \frac{p^4}{n^2}\, \frac{e^2}{\pi}\{\frac{1}{n^2\,p^3}\mathbbm{1}_{\{i\neq k,l\neq m\}}+\frac{1}{n\,p^3}\mathbbm{1}_{\{i=k,l\neq m\}}+\frac{1}{p^3}\mathbbm{1}_{\{i= k,l=m\}}\}(1+o(1))\,,\label{rateA2}
\end{flalign}
where $o(1)$ indicates a numerical sequence converging to zero, as $n\uparrow\infty$.
\end{prop}

\subsubsection{Proof of Theorem \ref{WignerTHM} assuming Proposition \ref{A1} and Proposition \ref{A2}}

Simply plugging the results from Proposition \ref{A1} and Proposition \ref{A2} into \eqref{final}, we have that
\begin{flalign*}
&d(F_n,N)^2\leq\,\pi^2\,  \sum_{i,k,l,m=1}^n\(\frac{4^4\,p^4}{n^4}\sum_{Q_1,Q_2=0}^{2p-4}\,\(Q_1+1\)\(Q_2+1\)\times\right.\\
&\times\{\[\frac{e}{2\,\sqrt{2\,\pi\,p^3}}\mathbbm{1}_{\{i=l\}}+\frac{e}{2\,n\,\sqrt{2\,\pi\,p^3}}\mathbbm{1}_{\{i\neq l\}}\]\mathbbm{1}_{\{Q_1=Q_2=0\}}+\right.\\
&\qquad\quad+2\,\[\frac{e^2}{\sqrt{2\,\pi^2\,p^3\, Q^3}}\mathbbm{1}_{\{i= l,k=m\}}+\frac{2\,e^2}{n\,\sqrt{2\,\pi^2\,p^3\, Q^3}}\mathbbm{1}_{\{i\neq l,k=m\}}+\right.\\
&\qquad\qquad\qquad\left.+\frac{e^2}{n^2\,\sqrt{2\,\pi^2\,p^3\, Q^3}}\mathbbm{1}_{\{i\neq l,k\neq m\}}\]\mathbbm{1}_{\{Q_1, Q_2 \text{ even, } Q_1\neq0\}}+\\
&\left.\left.\qquad\qquad\qquad\qquad\qquad\qquad+\[\frac{e^2}{n^2\,\sqrt{2\,\pi^2\,p^3\, Q^3}}\]\mathbbm{1}_{\{Q_1, Q_2 \text{ odd}\}}\}\)^{1/2}\,\times\\
&\qquad\times\,\(\frac{p^4}{n^2}\, \frac{e^2}{\pi}\{\frac{1}{n^2\,p^3}\mathbbm{1}_{\{i\neq k,l\neq m\}}+\frac{1}{n\,p^3}\mathbbm{1}_{\{i=k,l\neq m\}}+\frac{1}{p^3}\mathbbm{1}_{\{i= k,l=m\}}\}\)^{1/2}(1+o(1))\nonumber\\
&\leq \,16\,\pi^2\, \sum_{i,k,l,m=1}^n\(\frac{p^4}{n^4}\frac{e}{2\,n\,\sqrt{2\,\pi\,p^3}}\mathbbm{1}_{\{i\neq l\}}+\frac{4\,p^6}{n^4}\sum_{Q_1,Q_2=1}^{2p-4}\frac{e^2}{n^2\,\sqrt{2\,\pi^2\,p^3}}\mathbbm{1}_{\{i\neq l,k\neq m\}}\)^{1/2}\,\times\\
&\qquad\qquad\qquad\qquad\times\,\(\frac{p^4}{n^2}\frac{e^2}{\pi}\frac{1}{n^2\,p^3}\mathbbm{1}_{\{i\neq k,l\neq m\}}\)^{1/2}(1+o(1))\nonumber\\
&\leq \,16\,\pi^2\,\(\frac{e^{3/2}}{(2\,\pi)^{3/4}} \,\frac{p^{7/4}}{\sqrt{\,n\,}} + \frac{4\,e^2}{2^{1/4}\,\pi} \,\frac{p^{15/4}}{\,n\,}\)\,(1+o(1))\,\, .
\end{flalign*}
Now, it is straightforward to see that the bound goes to zero if $p=o(n^{4/15})$. 

Moreover, if $p=o(n^{1/4})$, then
$$
\frac{2\,e}{2^{1/8}\,\sqrt{\pi}} \,\frac{p^{15/8}}{\sqrt{\,n\,}}=o\(\frac{e^{3/4}}{(2\,\pi)^{3/8}} \,\frac{p^{7/8}}{n^{1/4}}\),
$$
while if $p=O(n^\alpha)$, with $\alpha \in [1/4,15/8)$, then
$$
\frac{e^{3/4}}{(2\,\pi)^{3/8}} \,\frac{p^{7/8}}{n^{1/4}}=o\(\frac{2\,e}{2^{1/8}\,\sqrt{\pi}} \,\frac{p^{15/8}}{\sqrt{\,n\,}}\);
$$
as a consequence, equation \eqref{cases-Wig} is proved and the theorem is established.
\hfill \mbox{\raggedright \rule{0.1in}{0.1in}}


\section*{Acknowledgements} 

This work has been extracted from the author's PhD thesis, defended on November 29th 2018 at the University of Luxembourg, under the supervision of Giovanni Peccati. The author would like to thank him for the constant help and support. The author also thanks Christian D\"obler, Maurizia Rossi and Guangqu Zheng for useful discussions. Finally, the author thanks an anonymous referee that with his/her insightful and careful comments considerably improved the present paper.

This work was part of the AFR research project {\it High-dimensional challenges and non-polynomial transformations in probabilistic approximations} (HIGH-NPOL) funded by FNR -- Luxembourg National Research Fund.


\bibliographystyle{alpha}

\bibliography{Bibliography}


\end{document}